\documentclass[english,11pt, reqno,oneside]{amsart}
\usepackage{amsmath,amsthm,amsfonts,amssymb,mathrsfs,bm,graphicx,stmaryrd}
\usepackage{mathtools}
\usepackage[%
]{caption}
\usepackage[list=true]{subcaption}
\usepackage{dsfont}
\usepackage{multicol}
\usepackage[colorlinks=true,linkcolor=blue,citecolor=blue]{hyperref}
\usepackage{bbm}
\newcommand{\fD}{\mathfrak{D}}

\newcommand{\ZZ}{\mathbb{Z}}

\newcommand{\QQ}{\mathbb{Q}}

\newcommand{\RR}{\mathbb{R}}

\newcommand{\PP}{\mathbb{P}}
\newcommand{\EE}{\mathbb{E}}

\newcommand{\fz}{\mathfrak{z}}

\newcommand{\NN}{\mathbb{N}}
\usepackage[letterpaper,hmargin=1.0in,vmargin=1.0in]{geometry}
\parskip 	\smallskipamount

\newtheorem{theorem}{Theorem}
\newtheorem{lemma}[theorem]{Lemma}

\newtheorem*{conjecture}{Problem}

\newtheorem{corollary}[theorem]{Corollary}
\newtheorem{proposition}[theorem]{Proposition}

\theoremstyle{definition}
\newtheorem{remark}[theorem]{Remark}

\newcommand{\cR}{\mathcal{R}}
\newcommand{\cP}{\mathcal{P}}
\newcommand{\cA}{\mathcal{A}}
\newcommand{\fB}{\mathfrak{B}}
\newcommand{\cB}{\mathcal{B}}
\newcommand{\tcR}{\widetilde{\mathcal{R}}}
\newcommand{\cC}{\mathcal{C}}
\newcommand{\cD}{\mathcal{D}}

\newcommand{\wcC}{\widetilde{\mathcal{C}}}

\newcommand{\cM}{\mathcal{M}}
\newcommand{\cL}{\mathcal{L}}

\newcommand{\supp}{\mathrm{supp}}

\newcommand{\cS}{\mathcal{S}}

\newcommand{\cX}{\mathcal{X}}

\newcommand{\dis}{\mathbf{Disjt}}
\newcommand{\mult}{\mathbf{Mult}}

\DeclareMathOperator*{\argmax}{\arg\! \max}

\newcommand{\cH}{\mathcal{H}}

\usepackage[backend=biber,doi=false,style=alphabetic,url=false,maxalphanames=10,maxnames=50]{biblatex}
\AtBeginBibliography{\small}

\begin{document}

 \title[]{\large Atypical stars on a directed landscape geodesic}
\author[]{Manan Bhatia}
  
\address{Manan Bhatia, Department of Mathematics, Massachusetts Institute of Technology, Cambridge, MA, USA}
\email{mananb@mit.edu}
\date{}
\begin{abstract}
  \normalsize
  In random geometry, a recurring theme is that any two geodesics emanating from a typical point part ways at a strictly positive distance from the above point, and we call such points as $1$-stars. However, the measure zero set of atypical stars, the points where such coalescence fails, is typically uncountable and the corresponding Hausdorff dimensions of these sets have been heavily investigated for a variety of models including the directed landscape, Liouville quantum gravity and the Brownian map.  In this paper, we consider the directed landscape-- the scaling limit of last passage percolation as constructed in the work Dauvergne-Ortmann-Vir{\'a}g '18-- and look into the Hausdorff dimension of the set of atypical stars lying on a geodesic.
  We show that the above dimension is almost surely equal to $1/3$. This is in contrast to Ganguly-Zhang '22, where it was shown that set of atypical stars on the line $\{x=0\}$ has dimension $2/3$. This reduction of the dimension from $2/3$ to $1/3$ yields a quantitative manifestation of the smoothing of the environment around a geodesic with regard to exceptional behaviour.
\end{abstract}
\maketitle

\begin{figure}
  \centering
  \captionsetup[subfigure]{labelformat=empty}
  \subcaptionbox{}{\includegraphics[width=0.42\textwidth]{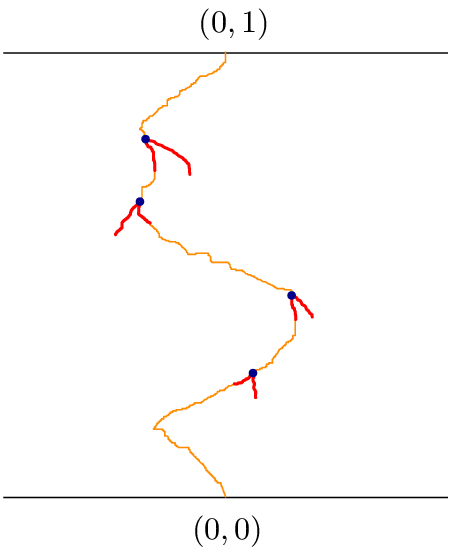}}
  \hfill
  \subcaptionbox{}{\includegraphics[width=0.42\textwidth]{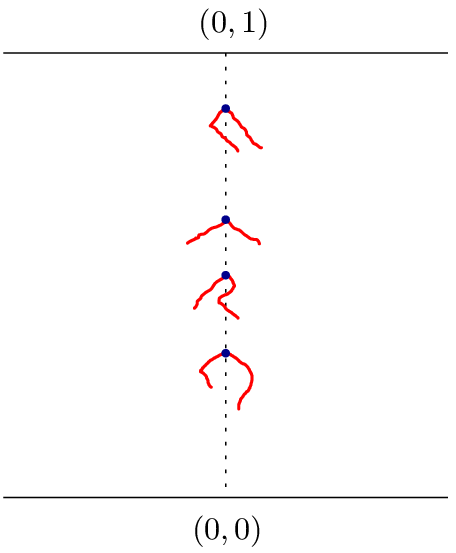}}
  \caption{The panel to the left shows some atypical stars on the geodesic $\Gamma$ from $(0,0)$ to $(0,1)$ and the main result of this work shows that such atypical stars have temporal dimension $1/3$. In contrast, for atypical stars on a vertical line (see right panel), the corresponding dimension is $2/3$ as proved in \cite{GZ22}.}
\end{figure}
\newpage
\section{Introduction}
\label{sec:intro}

The past few years have witnessed tremendous progress in the study of certain natural continuum models of random geometry-- the directed landscape \cite{DOV18,Gan21}, the Liouville quantum gravity metrics \cite{DDDF20,GM21,DDG21} and the special case of the Brownian map \cite{LeGal19}. While the origins of these objects consist of an array of seemingly disparate areas like Kardar-Parisi-Zhang universality, conformal field theory, and random planar maps, these models are strikingly similar in some aspects, and one important such aspect is the phenomenon of coalescence/confluence \cite{RV21,GZ22,GM20,LeGal10,BSS19} of geodesics. Intuitively, due to enough independence in different regions of space, there emerge certain attractive `highways', and typically geodesics attempt to merge into these relatively small number of highways to optimize their lengths. An important question which has attracted intense activity recently \cite{MSZ21,DSV22,Bat20,JLS20,BBG21} is to zoom in around a geodesic and attempt to describe the environment around it; note that this would be very different from the environment around a uniformly random point in the space since the geodesic goes through highly atypical regions of the space. Owing to the length optimizing nature of a geodesic, it is a general belief that the former environment should be ``smoother'' than the latter, in the sense of being comparatively more concentrated and displaying a reduced amount of exceptional behaviour. The aim of this work is to demonstrate the above phenomenon for the directed landscape in a certain quantitative sense.

Before introducing the directed landscape, we first introduce exponential last passage percolation (LPP), an integrable discrete model on the lattice which converges to the directed landscape in the fine mesh limit.
Consider the planar lattice and endow each vertex with an i.i.d.\ $\mathtt{exp}(1)$ weight. For any two points $u\leq v$, meaning that the above inequality holds coordinate-wise, the passage time from $u$ to $v$ is defined to be the largest possible weight of any up-right path from $u$ to $v$, where the weight of a path is the sum of the weights of all vertices it passes through. Apart from the passage times themselves, one can also look at the paths of largest weight from $u$ to $v$ as $u$ and $v$ vary, and these paths are called geodesics. %

In contrast to the discrete LPP, the directed landscape $\cL$ is a random real valued function defined on the space
\begin{equation}
  \label{eq:21}
  \RR^4_{\uparrow}=\{(x,s;y,t)\in \RR^4: s<t\}
\end{equation}
with $\cL(x,s;y,t)$ to be interpreted as the last passage time from the point $(x,s)$ to $(y,t)$ with $s<t$. Analogously to a composition law satisfied by LPP, the directed landscape satisfies
\begin{equation}
  \label{eq:22}
  \cL(x,s;y,t)=\max_{z\in \RR}(\cL(x,s;z,r)+\cL(z,r;y,t))
\end{equation}
for any $x,s,y,r,t\in \RR$ with $s<r<t$. Now for any path $\gamma\colon [s,t]\rightarrow\RR$, its weight $\ell(\gamma)$ can be defined as
\begin{equation}
  \label{eq:23}
  \ell(\gamma)\coloneqq \inf_{k\in \NN}\inf_{s=t_0<t_1<\cdots<t_k=t}\sum_{i=1}^k\cL(\gamma(t_{i-1}),t_{i-1};\gamma(t_i),t_i)
\end{equation}
and one can consider all the paths $\gamma\colon [s,t]\rightarrow \RR$ satisfying $\gamma(s)=x$ and $\gamma(t)=y$, and then maximize $\ell(\gamma)$ over all such paths. Any path achieving the above maximum is called a geodesic and is denoted by $\Gamma_{(x,s)}^{(y,t)}$; it is possible to show that for any fixed $(x,s;y,t)\in \RR_\uparrow^4$, there is a.s.\ a unique such geodesic $\Gamma_{(x,s)}^{(y,t)}$ \cite[Theorem 1.7]{DOV18}. Further, in a striking phenomenon known as \textit{coalescence} \cite{GZ22,RV21} (see \cite{BSS19} for the original coalescence statement for LPP), for any fixed $(y,t)$, we a.s.\ have that any geodesics $\Gamma^{(y,t)}_{(x_1,s_1)}$ and $\Gamma^{(y,t)}_{(x_2,s_2)}$ for any $x_1,x_2\in \RR$ and $s_1,s_2<t$ merge together at a time strictly smaller than $t$. However, it is still possible that there exist random $(y,t)$ for which the above phenomenon does not occur, and it is not difficult to see that such atypical points must exist almost surely.

Investigating the fractal properties of the exceptional points where there is locally more than one geodesic emanating out has received considerable interest recently across a variety of areas encompassing random geometry, and we now give a flavour of these developments \cite{Ham20,BGH21,BGH19,GZ22,LeGal22,MQ20,AKM17,Gwy21}. In the field of Brownian geometry, the first work involving an extensive analysis of this exceptional set was \cite{Mie13}, where an understanding of this atypical set led to the proof of the convergence of uniform planar quadrangulations to the Brownian sphere. The recent works \cite{LeGal22, MQ20} go further and show that the dimension of the set of points with $k$ geodesic arms emanating out locally is exactly $5-k$ for $k\leq 5$, with the set being empty for $k\geq 6$. For the case $k=2$, it has been shown in \cite{AKM17} that the above set is one of first Baire category. In the more general model of $\gamma$-Liouville quantum gravity, \cite{Gwy21} considers the closely related set of points having two geodesics to a fixed point and shows that this set is dense with dimension at most $d_\gamma-1$, with $d_\gamma$ being the dimension of the ambient space. For the directed landscape, the model on which this paper is based, a detailed discussion of the recent developments will appear in the upcoming Section \ref{sec:main-results} and thus we refrain from discussing these now.  An accessible introduction to the developments around fractality in the directed landscape can be found in the recent survey paper \cite{GH22}.

A novel theme in this work is to study such fractal properties, but for the directed landscape rooted on geodesics instead. The environment around a geodesic was identified in the work \cite{DSV22} (see \cite{MSZ21} for a corresponding result in LPP) and as mentioned earlier, it is a general belief that this environment should be ``smoother'' than the unconditional environment (see \cite{JLS20} for a version of this in first passage percolation). In this work, we demonstrate a version of the above principle by doing the first investigation into the fractal geometry of exceptional points on a geodesic, which we emphasize, is a set entirely composed of atypical points itself. More specifically, a consequence of the work \cite{GZ22} is that on a fixed space line, say $\{x=0\}$, the set of points which locally have more than one geodesic emanating out in the negative time direction, has dimension exactly $2/3$. In contrast, we look at the exact same set of points but instead on a geodesic %
$\Gamma_{(0,0)}^{(0,1)}$ and show that its dimension reduces to $1/3$ from the earlier value $2/3$. This reduction in dimension shows that with regard to the quantity of exceptional points of a certain type, geodesics do in fact become smoother. %

\subsection{The main result}
\label{sec:main-results}

The primary result of this work revolves around a certain type of exceptional points called \text{atypical stars}, which we now define. As a consequence of the coalescence properties \cite[Lemma 4.1]{GZ22} of geodesics in the directed landscape, for each fixed point $(x,s)$, there is locally only one geodesic emanating out in the forward time direction and similarly only one geodesic coming out in the reverse time direction. To be precise, we define a point $(y,t)$ to be a reverse $1$-star if for any two points $(x_1,s_1)$ and $(x_2,s_2)$ with $s_1,s_2<t$, all geodesics $\Gamma_{(x_1,s_1)}^{(y,t)},\Gamma_{(x_2,s_2)}^{(y,t)}$ have the property that there exists a $\delta>0$ such that for all $r\in (t-\delta,t]$,
\begin{equation}
  \label{eq:24}
  \Gamma_{(x_1,s_1)}^{(y,t)}(r)=\Gamma_{(x_2,s_2)}^{(y,t)}(r).
\end{equation}

A definition analogous to the above can be used to define forward $1$-stars as well. The following lemma records the above-mentioned consequence of geodesic coalescence.

\begin{lemma}
  \label{lem:main:4}
Any fixed point $(x,s)$ in the directed landscape is almost surely both a forward and a reverse $1$-star. Borrowing terminology from the literature on Brownian geometry \cite{Mie13,AKM17,MQ20,LeGal22}, we call such points as \textbf{$1$-stars}.
\end{lemma}
However, it is still possible that there exist a random set of points $(x,s)$ which are not $1$-stars and the fractal properties of these has recently been an active area of research; we refer to such points as \textbf{atypical stars}. Though the above notion of atypical stars has not been considered in the directed landscape literature so far, a series of recent works \cite{BGH19,BGH21,GZ22} have looked at sets of points having disjoint geodesics to two fixed points in the plane. It is not difficult to see that the set of atypical stars is the countable union of these when the two fixed points are varied over the rationals. The advantage of considering the set of atypical stars is that it is a natural object associated to the landscape $\cL$ and does not involve fixing two arbitrary fixed points to consider disjoint geodesics to. We now give a quick tour through the above-mentioned recent results, and frame them in terms of atypical stars to suit our setting.

The first pair of relevant works \cite{BGH21,BGH19} concern the set of atypical stars $(x,s)$ when $s$ is fixed and $x$ allowed to vary. The analysis of this set involved the consideration of a %
novel object called the (uncentered) difference profile
\begin{equation}
  \label{eq:128}
  \fD^{\mathrm{un}}(x)=\cL(1,0;x,1)-\cL(0,0;x,1)
\end{equation}
which is surprisingly, a monotonic function of $x$. Note that later (see Section \ref{ss:diff}) in the paper, we define certain difference profiles which are rooted around a random point and thus to avoid potential confusion, we use the superscript ``un'' here to indicate that the above difference profile is unrooted.

In the argument in \cite{BGH21}, %
points $(x,1)$ admitting almost disjoint geodesics (meaning, disjoint except possibly at their endpoints) to $(0,0)$ and $(1,0)$ were shown to be the support of the Lebesgue-Stieltjes measure corresponding to $\fD^{\mathrm{un}}$ and then the H\"older $1/2-$ nature of $\fD^{\mathrm{un}}$ was used to obtain a lower bound of $1/2$ on the dimension of atypical stars. On the other hand, the corresponding upper bound of $1/2$ was obtained by using results on the rarity of disjoint geodesics from \cite{Ham20} which themselves relied on the \textit{Brownian Gibbs property} of \cite{CH14}. %
We now record the result described in this paragraph, adapted to the setting of atypical stars, as a proposition for future reference.

\begin{proposition}[\cite{BGH21},{\cite[Theorem 1.9]{BGH19}}]
  \label{prop:2}
  For any fixed $s$, the set of $x\in \RR$ such that $(x,s)$ is an atypical star a.s.\ has Hausdorff dimension $1/2$.
\end{proposition}

In the recent work \cite{GZ22}, the points $(x,s)$ admitting almost disjoint geodesics to $(-1,0),(1,0)$ were considered, now on a vertical space line instead of a horizontal time line and were shown to have dimension $2/3$. We record the results from \cite{GZ22}, adapted to the setting of atypical stars, in the following proposition.

\begin{proposition}[{\cite[Theorem 2.2, Theorem 2.1]{GZ22}}]
  \label{prop:3}
  For any fixed $x$, the set of $s\in \RR$ such that $(x,s)$ is an atypical star a.s.\ has Hausdorff dimension $2/3$. Further, the set of $(x,s)\in \RR^2$ which are atypical stars a.s.\ has Hausdorff dimension $5/3$.
\end{proposition}
As is described in the introduction of \cite{GZ22}, if one tries to use the temporal difference profile $s\mapsto \cL(1,0;0,s)-\cL(-1,0;0,s)$ along with a H\"older continuity argument to obtain the lower bound for the first statement in the above proposition, one only obtains a lower bound of $1/3$ instead of the correct value $2/3$, with the reason behind this discrepancy being that while the spatial difference profile $\fD^{\mathrm{un}}$ is monotonic, the above temporal version is not. Thus there are some cancellations which ultimately lead to a square-root boosting effect of the dimension from $1/3$ to $2/3$. A more detailed discussion around around this point is delegated to Remark \ref{rem:int}. We are now ready to state the central result of this paper.

\begin{theorem}
  \label{thm:2}
  Let $\Gamma$ be the a.s.\ unique geodesic from $(0,0)$ to $(0,1)$ and consider the set of times $t\in (0,1)$ for which $(\Gamma(t),t)$ is an atypical star. Then this set almost surely has Hausdorff dimension $1/3$. %
\end{theorem}
Thus on comparing with Proposition \ref{prop:3}, we see that the dimension starkly drops from $2/3$ to $1/3$ if one looks for atypical stars along a geodesic instead of along a fixed vertical line, and this is a quantitative manifestation of the environment around a geodesic being ``smoother'' than typical. On the way to establishing Theorem \ref{thm:2}, we obtain another independently interesting result about the interaction of semi-infinite geodesics of different angles, which we now introduce (refer to Section \ref{sec:busem} for more details). By the results from \cite{BSS22,RV21,GZ22}, almost surely, for every $\theta\in \RR$ and point $p=(x,s)\in \RR^2$, %
there exists a path $\Gamma^{p}_\theta\colon(-\infty,s]\rightarrow \RR$ such $\Gamma^{p}_\theta(s)=x$ and any restriction of $\Gamma^{p}_\theta$ to any finite interval is a geodesic for $\cL$ and further, $\lim_{t\rightarrow \infty}\Gamma^{p}_\theta(-t)/t=\theta$. We refer to $\Gamma^p_{\theta}$ as a downward $\theta$-directed semi-infinite geodesic emanating from $p$ and note that it is a.s.\ unique for fixed values of $p,\theta$. Employing the abbreviation $\Gamma_\theta=\Gamma^{(0,0)}_\theta$, we have the following result.
\begin{theorem}
  \label{thm:1}
  For any fixed $\theta\neq 0$, consider the set of $t\in (-\infty,0)$ which admit a pair of almost disjoint geodesics $\Gamma_0^{(\Gamma_0(t),t)}$ and $\Gamma_\theta^{(\Gamma_0(t),t)}$. Then this set a.s.\ has Hausdorff dimension $1/3$.
\end{theorem}
The techniques used in this work build upon the series of works \cite{Ham20,BGH21,BGH19,GZ22,BSS22}. As in the first three of the above works, the crucial technique for obtaining dimension lower bounds is to study the difference profile, but this time along a geodesic itself instead of along fixed space lines or time lines. The lower bound obtained from the difference profile argument is actually optimal with the moral reason for this being a hidden new monotonicity present in the difference profile as the point on the geodesic is varied. The novel aspect in the proof of the upper bound is to build upon the study of disjoint geodesics but now in conjunction employ the Brownian-Bessel description (see \cite{DSV22}) of the weight profile around a point on the geodesic along with the recently demonstrated \cite{BSS22} integrability of the spatial Busemann process of the directed landscape.

\paragraph{\textbf{Notational comments}}
We will often work with geodesics between different points in the directed landscape, and in case there are two geodesics which are disjoint except possibly at their endpoints, we call them almost disjoint. %
For $x<y\in \RR$, we use the notation $[\![x,y]\!]=[x-1,y+1]\cap \ZZ$. Throughout this paper, we will work with the following processes and all of them are taken to be mutually independent. The letters $B,R$ denote a standard Brownian motion and an independent standard Bessel-$3$ process, and we further use $\{B_i\}_{i=0}^{\infty}$ for i.i.d.\ standard Brownian motions. Finally, we always use $\Gamma$ to denote the a.s.\ unique geodesic from $(0,0)$ to $(0,1)$ in the directed landscape $\cL$.
\paragraph{\textbf{Acknowledgements}} The author thanks Riddhipratim Basu, Shirshendu Ganguly, Milind Hegde %
for the comments and especially I-Hsun Chen for the help with simulations. The author acknowledges the support of the Peterson Fellowship at MIT along with the NSF grants DMS-1712862 and DMS-1953945, and thanks the Institute for Advanced Study for its support as well.
\section{Ideas of the proof of Theorem \ref{thm:2}}
\label{sec:ideas}
\subsubsection*{\textbf{Introducing the intermediate step -- atypical stars on semi-infinite geodesics}}
An important step in the proof of Theorem \ref{thm:2} is to compute the dimension of atypical stars on a semi-infinite geodesic as an intermediate step to obtaining the same for a finite geodesic. This is required mainly for the upper bound, and the reason for doing this is related to the fact that strong $L^p$ estimates for the Brownianity of the KPZ fixed point \cite{MQR21} are not currently known; see Remark \ref{rem:inf} for a discussion. For a quick introduction to the Busemann functions associated to semi-infinite geodesics, we refer the reader to  Section \ref{sec:busem}. As a summary, %
for a fixed $\theta$ and for $t\in \RR$, we define the Busemann function $\cB^t_\theta(x)=\cL(\fz;x,t)-\cL(\fz;0,t)$ with $\fz$ denoting the coalescence point of $\Gamma^{(x,t)}_\theta$ and $\Gamma^{(0,t)}_\theta$. %

The goal now is to obtain the dimension of atypical stars on the a.s.\ unique $0$-directed semi-infinite geodesic $\Gamma_0$. To set up notation we define the set $\mult_\theta$ for $\theta\in \RR$ by
\begin{equation}
  \label{eq:6}
  \mult_\theta=
  \left\{
    t\in (-\infty,0):\exists \text{ almost disjoint geodesics } \Gamma^{(\Gamma_0(t),t)}_0,\Gamma^{(\Gamma_0(t),t)}_\theta 
  \right\}.
\end{equation}

By using the flip symmetry of the landscape \cite[(3) in Lemma 10.2]{DOV18} and arguing that each point $(\Gamma_0(t),t)$ which is not a reverse $1$-star must satisfy $t\in \mult_\theta$ for all some $\theta$, it is not difficult to reduce to showing that for any fixed $\theta>0$, $\dim \mult_\theta=1/3$ almost surely.

\subsubsection*{\textbf{The lower bound via the $1/3-$ H\"older continuity of the geodesic difference profile}}

For this part of the proof, we consider a novel object, which we call the geodesic difference profile (see Section \ref{ss:diff} for more details) and define by
\begin{equation}
  \label{eq:31}
  \cD_\theta(t)=\cB_\theta((\Gamma_0(t),t))-\cB_0((\Gamma_0(t),t)).
\end{equation}
 This difference profile is well-defined only modulo a global additive constant, and note that this is very different from the spatial difference profile $\fD^{\mathrm{un}}$ defined earlier. Similar to $\fD^{\mathrm{un}}$, $\cD_\theta$ also increases monotonically (Lemma \ref{lem:10}) as $t$ increases and further, by coalescence, the Lebesgue-Stieltjes measure $\mu_\theta$ of $\cD_\theta$ satisfies $\supp \mu_\theta \subseteq \mult_{\theta}$ (Lemma \ref{lem:11}). The desired $1/3$ lower bound on the dimension is now obtained by using the $1/3-$ H\"older continuity of the directed landscape in the temporal direction to obtain $\dim \supp \mu_\theta \geq 1/3$ a.s.\ just as in \cite{BGH21}. Though the proof of the lower bound above was immediate upon the introduction of the geodesic difference profile, the upper bound is much harder.

 \subsubsection*{\textbf{An important event for the proof of the upper bound -- $\dis_\theta\{J_{i,\delta}\}$}}
 For a set $K\subseteq \RR^2$ and $\theta>0$, we define the event
\begin{equation}
  \label{eq:41}
  \dis_\theta\{K\}=
  \left\{
    \exists ~p \in K \text{ admitting almost disjoint geodesics } \Gamma^p_0,\Gamma^p_\theta 
  \right\}.
\end{equation}
For our setting, we will use the above event with $K=J_{i,\delta}$, where the latter is defined for all negative $i\in \ZZ$ by
\begin{equation}
  \label{eq:29}
  J_{i,\delta}=[\Gamma_0(i\delta^{3/2})-\delta\log^4 \delta^{-1},\Gamma_0(i\delta^{3/2})+\delta\log^4 \delta^{-1}]\times \{i\delta^{3/2}\}.
\end{equation}
\begin{figure}
  \centering
  \includegraphics[width=0.4\textwidth]{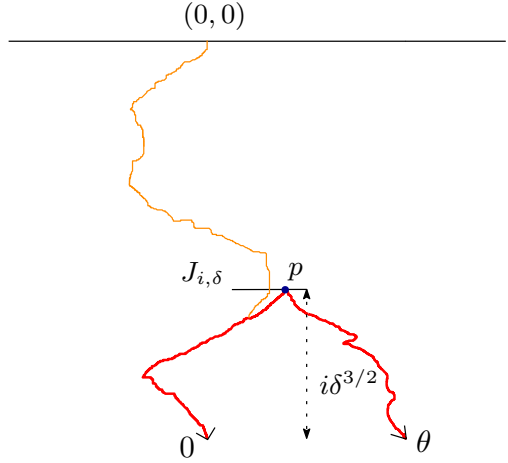}
  \caption{The event $\dis_\theta\{J_{i,\delta}\}$ demands that there exist two almost disjoint downward semi-infinite geodesics $\Gamma^p_0,\Gamma^p_\theta$ from some point $p$ in the interval $J_{i,\delta}$.}
  \label{fig:disjt}
\end{figure}
One can now show (see Lemma \ref{lem:2}) that for any $s$ which is bounded away from $0$ and $-\infty$, if there exists a $t\in (s,s+\delta^{3/2})$ such that $t\in \mult_\theta$, then with high probability, $\dis_\theta\{J_{i,\delta}\}$ occurs for $i=\lceil s\delta^{-3/2}\rceil$ (see Figures \ref{fig:disjt}, \ref{fig:coal_trans}). For this reason, obtaining precise estimates on $\dis_\theta\{J_{i,\delta}\}$ is of central importance to the proof of the upper bound. While optimal estimates on the probabilities of $\dis_\theta\{K\}$ for a fixed spatial interval $K$ of size $\delta$ can be obtained using the arguments in \cite{Ham20}, our setting is different in that we require $K$ to instead be rooted at a point on the geodesic $\Gamma_0$ itself.

\subsubsection*{\textbf{Bounding $\dim\mult_\theta$ assuming $\PP(\dis_\theta\{J_{i,\delta}\})=\delta^{1-o(1)}$}}
\label{sec:bound-dimm-assum}

We will use the behaviour of the environment around a geodesic to obtain a $O(\delta^{1-o(1)})$ upper bound
for the probability of  $\dis_\theta\{J_{i,\delta}\}$ for $i\delta^{3/2}$ bounded away from $0$ and $-\infty$. This should be contrasted with \cite{BGH21} which yields an $O(\delta^{1/2-o(1)})$ bound on the corresponding probability when the interval $J_{i,\delta}$ around the geodesic is replaced by a fixed spatial interval of size $\delta$. We shall describe how to obtain the above $O(\delta^{1-o(1)})$ bound shortly, but we first describe how to complete the proof of the upper bound if we assume this; analogous techniques were used to obtain the upper bound in \cite{GZ22} as well. If we take $\varepsilon= \delta^{3/2}$, then by what we described in the previous paragraph, for any $\alpha>1$ and for each $s\in [-\alpha,-\alpha^{-1}]$, the following inclusion holds for each $\varepsilon>0$ on a high probability regularity event,
\begin{equation}
  \label{eq:27}
 \{ \mult_\theta\cap [s,s+\varepsilon]\neq \phi\}\subseteq \dis_\theta\{J_{i,\varepsilon^{2/3}}\}.
\end{equation}
By the $O(\delta^{1-o(1)})$ assumption above, the event on the right has probability bounded by $O(\varepsilon^{2/3-o(1)})$.
On now dividing $[-\alpha,-\alpha^{-1}]$ into segments of length $\varepsilon$ each, we obtain that the expected number of such segments needed to cover $\mult_\theta\cap [-\alpha,-\alpha^{-1}]$ is $O(\varepsilon^{-1}\times\varepsilon^{2/3-o(1)})=\varepsilon^{-1/3-o(1)}$ and this implies a $1/3$ upper bound for the dimension of $\mult_\theta\cap [-\alpha,-\alpha^{-1}]$. By taking a countable union along an increasing sequence of $\alpha$s, one obtains the corresponding bound for $\mult_\theta$ itself.

\subsubsection*{\textbf{The crucial $\delta^{1-o(1)}$ upper bound on $\PP(\dis\{J_{i,\delta}\})$}}
\label{sec:crucial-delta1-o1}

We now describe how to obtain the $\delta^{1-o(1)}$ upper bound on the probability of $\dis_\theta\{J_{i,\delta}\}$ for $i\delta^{3/2}$ bounded away from $0$ and $-\infty$. Recall the objects $B,R,\{B_i\}_{i\geq 0}$ from the notational comments just before Section \ref{sec:ideas}; we couple these to be mutually independent.
As shown in the recent works \cite{Bus21,SS21,BSS22}, the joint law of the spatial Busemann process $\cB^t_\theta(x)$ in $\theta$ and $x$ has an explicit integrable description in terms of independent Brownian motions with drift (see Section \ref{sec:integ}). For us, the relevant fact (Proposition \ref{prop:1}) is that there exists a process $\cP^{t,\theta}=(\cP^{t,\theta}_1,\cP^{t,\theta}_2)$ measurable with respect to $\sigma(\cL)$ such that $(\cP^{t,\theta}_1(x),\cP^{t,\theta}_2(x))\stackrel{d}{=}(\sqrt{2}B_1(x),\sqrt{2}B_2(x)+2\theta x)$ as a process in $x$ and moreover, one can express the uncentered Busemann spatial difference profile $\cD^{\mathrm{un},t}(x)=\cB_\theta^t(x)-\cB_0^t(x)$ as
\begin{equation}
  \label{eq:129}
  \cD^{\mathrm{un},t}(x)=\sup_{y\leq x}(\cP_2^{t,\theta}(y) -\cP^{t,\theta}_1(y))+K,
\end{equation}
where $K$ is a random constant given by $K=-\sup_{y\leq 0}(\cP_2^{t,\theta}(y)-\cP_1^{t,\theta}(y))$.
We note that expressions similar to the above were used in \cite{GH21} and are easy to discern by interpreting $\cB_\theta^t$ as a Brownian last passage value across the $2$-line ensemble $\cP^{t,\theta}$; we refer the reader to \cite[Remark 1.3]{Dau21} for this alternate interpretation of the integrability of the spatial Busemann process.

We now give a quick heuristic overview of a proof of the upper bound in Proposition \ref{prop:2} using the argument from \cite{GH21}, and this will be helpful since the proof of our main result builds upon this argument. Since a Brownian motion with drift is locally Brownian, the pair $(\cP_1^{t,\theta},\cP_2^{t,\theta})$ is locally absolutely continuous to a pair of independent Brownian motions of diffusivity $2$, and thus the increments of $\cP_2^{t,\theta}-\cP^{t,\theta}_1$ are locally absolutely continuous to a Brownian motion of diffusivity $4$. Noting that points $x$ admitting almost disjoint semi-infinite geodesics $\Gamma^{(x,t)}_0,\Gamma^{(x,t)}_\theta$ must be points of increase of the monotonic function $\cD^{\mathrm{un},t}$, it can be obtained by using \eqref{eq:129} that the relevant dimension is upper bounded by the dimension of the set of $x>0$ for which $B(x)=\sup_{0\leq y\leq x}B(y)$, and this can be shown to have dimension $1/2$.  Indeed, for an upper bound, it is not difficult to use the reversibility of Brownian motion to upper bound the probability of there existing such an $x$ in an interval $I=[b,b+\delta]$ with $b>1$ by the probability $\PP(\argmax_{x\in [0,1]}B(x)\in [0,\delta^{1-o(1)}])\leq \delta^{1/2-o(1)}$.

In contrast, for the proof of the upper bound of $\PP(\dis_\theta\{J_{i,\delta}\})$, we instead have to consider the local behaviour of the difference profile $\cD^{\mathrm{un},t}$ around the random location $\Gamma_0(t)$, and this is very different from its behaviour around (say) the point $0$. We now look into this behaviour by analysing both the terms inside the supremum in \eqref{eq:129}. First, the behaviour of $\cP_1^{t,\theta}$ is not Brownian around $\Gamma_0(t)$, but instead $x\mapsto \cP_1^{t,\theta}(\Gamma_0(t)+x)$ actually looks like $B-R$, where $R$ is a standard Bessel-$3$ process independent of $B$. Indeed, this is not difficult to predict (see also \cite{DSV22}), and can be seen by noting that for any fixed $t\in (-\infty,0)$, $\Gamma_0(t)=\argmax_{x\in \RR}\{\cB^t_0(x)+\cL(x,t;0,0)\}$, where the two summands, each considered as a process in $x$, are independent and locally absolutely continuous to $\sqrt{2}B$. Note that we use here that $\cL(x,t;0,0)$ being the $\mathrm{Airy}_2$ process in $x$ is locally Brownian \cite{CH14}. Now on using that a Brownian motion around its maximum behaves like the negative of a Bessel-$3$ process $R$ of the same diffusivity, one can obtain that $\sqrt{2}B_1=(\sqrt{2})^{-1}(B_1+B_2)+(\sqrt{2})^{-1}(B_1-B_2)$, looks like $B-R$, when viewed around the maximizer of $B_1+B_2$.

However, the behaviour of $\cP_2^{t,\theta}$ around $\Gamma_0(t)$ is still Brownian, the reason being that $\Gamma_0(t)\in \sigma (\cP^{t,\theta}_1,\cL(\cdot,t;0,0))$ and $\cP_2^{t,\theta}$ is independent of this $\sigma$-algebra. Thus rerooting $\cP_2^{t,\theta}$ around $\Gamma_0(t)$ does not destroy its locally Brownian nature. %
Using a similar argument as in the previous paragraph, the above assertions can be upgraded to the joint local absolute continuity of $(\cP_1^{t,\theta}(\Gamma_0(t)+x),\cP_2^{t,\theta}(\Gamma_0(t)+x))$ to $(B_1(x)-R(x),\sqrt{2}B_2(x))$ as processes in $x$, thereby yielding that $\cP_2^{t,\theta}(\Gamma_0(t)+x)-\cP^{t,\theta}_1(\Gamma_0(t)+x)$ is locally absolutely continuous to $\sqrt{3}B(x)+R(x)$. With $t=i\delta^{3/2}\in (-\infty,0)$, it can be shown that the occurrence of $\dis_\theta\{J_{i,\delta}\}$ implies the existence of a point of increase of $\cD^{\mathrm{un},t}$ in an interval of width $2\delta \log^4 \delta^{-1}$ around $\Gamma_0(t)$. By the above absolute continuity, \eqref{eq:129}, and some work, $\PP(\dis_\theta\{J_{i,\delta}\})$ can now be upper bounded by a term of the form $\PP(\argmax_{x\in [0,1]}(\sqrt{3}B+R)(x)\in [0,\delta^{1-o(1)}])$. To show that the above probability is $O(\delta^{1-o(1)})$, we recast it an estimate for the exit time of a Brownian motion started inside a right circular cone of angle $\pi/3$ and then use results from \cite{Bur77,DeB87,BS97,BCG83} to reduce to the computation of the first eigenvalue of the Laplacian on a $4$-dimensional spherical cap (see Section \ref{ss:cones}).
\subsubsection*{\textbf{An absolute continuity argument to transfer from semi-infinite geodesics to finite ones}}
\label{sec:an-absol-cont}

Though the above strategy yields the dimension of points on $\Gamma_0$ which are not reverse $1$-stars, it still remains to transfer the result to the case of the finite geodesic $\Gamma$. This is done by arguing that the law of a finite geodesic restricted to an interval is absolutely continuous to that of a semi-infinite geodesic restricted to the same interval. Thus the almost sure dimension result for atypical stars on $\Gamma_0$ transfers to the same for $\Gamma$. 

\noindent We have now approached the end of this section and make a couple of remarks before moving on.

\begin{remark}
  \label{rem:int}
  The intuitive reason for the tightness of the $1/3$ lower bound coming from the H\"older continuity of the geodesic difference profile $\cD_\theta$ is the monotonicity present in $\cD_\theta$ (Lemma \ref{lem:10}). The monotonicity implies that the contributions of different points in $\mult_{\theta}$ to the value $\cD_\theta(0)$ cannot cancel each other. Indeed, a similar phenomenon occurred in the work \cite{BGH21}, where the profile $x\mapsto \fD^{\mathrm{un}}(x)=\cL(1,0;x,1)-\cL(0,0;x,1)$ was seen to be increasing. Even here, the H\"older continuity argument yields the tight dimension lower bound of $1/2$ for the non-constancy points of the above function, which are the same \cite{BGH19} as the set of points $x$ admitting almost disjoint geodesics $\Gamma_{(0,0)}^{(x,1)},\Gamma_{(1,0)}^{(x,1)}$. As an example of a setting where monotonicity fails, one could consider the profile $t\mapsto \cL( 1,0;0,t)-\cL(-1,0;0,t)$. In this case, the H\"older continuity argument yields a lower bound of $1/3$ for the dimension of the set of $t$ such that there exist almost disjoint geodesics $\Gamma_{(0,0)}^{(0,t)},\Gamma_{(1,0)}^{(0,t)}$, but this is not optimal. Indeed, the above set was shown to almost surely have dimension $2/3$ in \cite{GZ22} as we stated in Proposition \ref{prop:3}.
\end{remark}
  \begin{remark}
    \label{rem:inf}
    One might wonder whether there is any need to introduce semi-infinite geodesics for the proof of Theorem \ref{thm:2}, since the difference profile $\fD^{\mathrm{un},t}(x)=\cL(1,0;x,t)-\cL(0,0;x,t)$ has an expression very similar to \eqref{eq:129} which has been used in the works \cite{GH22,Dau22}. The reason why this does not work is that the term corresponding to $\cP_2^{t,\theta}$ in the above-mentioned expression is a certain distance profile from $\infty$ in the Airy line ensemble to the second curve, and while this profile has been shown to be locally Brownian in \cite{GH22} using the Brownian absolute continuity of the KPZ fixed point \cite{SV21}, it is not known whether the associated Radon Nikodym estimates are in $L^p$ for all $p$. This lack of strong comparison precludes the use of the arguments using Brownianity described just after \eqref{eq:129}. In contrast, due to the integrability of the spatial Busemann process from \cite{Bus21,SS21,BSS22}, $\cP_2^{t,\theta}$ is simply a Brownian motion with drift and is thus strongly comparable to Brownian motion on compact intervals, and this allows the arguments to go through.
  \end{remark}
\section{Preliminaries}
\label{sec:collect}
\subsection{A brief introduction to the relevant geometric objects}
In this subsection, we describe some background material regarding the directed landscape which we shall require. For brevity, we only discuss the aspects which are immediately relevant to this work.

\subsubsection{\textbf{Semi-infinite geodesics and the Busemann process}}
\label{sec:busem}
Semi-infinite geodesics and their connection \cite{FP05,FMP09} to second class particles in the TASEP have been studied at length for the prelimiting LPP. In the works \cite{BSS22,RV21,GZ22}, this was extended to the limit to define the downward $\theta$-directed semi-infinite geodesics $\Gamma^p_\theta$ for $p\in \RR^2,\theta\in \RR$ which we defined just before stating Theorem \ref{thm:1}. %
An important property satisfied by semi-infinite geodesics is that for any fixed $\theta$, $\Gamma^{p_1}_\theta$ and $\Gamma^{p_2}_\theta$ started from any $p_1\neq p_2$ eventually coalesce. %
As mentioned earlier, we employ the abbreviation $\Gamma_\theta=\Gamma_\theta^{(0,0)}$ throughout the paper. Also, note that while we only define semi-infinite geodesics going downward in the time direction, there also similarly exist semi-infinite geodesics going upward in time.

Considering that the geodesics $\Gamma^\theta_p$ exist, one might hope to also associate corresponding passage times to the ``point at infinity'' in the $\theta$ direction. This can be done, modulo an additive constant via Busemann functions, a notion from geometry \cite{Bus12} originally introduced to first passage percolation in \cite{New95,Hof05}. Following \cite{BSS22}, for any $p_1\neq p_2\in \RR^2$ and any fixed $\theta\in \RR$, we define the Busemann function $\cB_\theta(p_1,p_2)$ by
\begin{equation}
  \label{eq:2}
  \cB_\theta(p_1,p_2)=\cL(\mathfrak{z};p_1)-\cL(\mathfrak{z};p_2),
\end{equation}
where $\mathfrak{z}$ is the coalescence point of $\Gamma^{p_1}_\theta$ and $\Gamma^{p_2}_\theta$; note that we might also consider $p\mapsto \cB_\theta(p)$ as a function which is well-defined modulo a global additive constant. We will primarily work in the setting where $p_2=(0,t)$ and $p_1=(x,t)$ and thus we simply abbreviate $\cB^t_\theta(x)=\cB_\theta((x,t),(0,t))$ and refer to the map $x\mapsto \cB^t_\theta(x)$ as the spatial Busemann process at time $t$.

\subsubsection{\textbf{Integrability of the spatial Busemann process}}
\label{sec:integ}
The work \cite{Bus21} constructed an object called the `stationary horizon' which represents the joint scaling limit of the Busemann process for Brownian LPP when the angle is allowed to vary arbitrarily but the point is allowed to vary only spatially. In the subsequent work \cite{BSS22}, it was shown for a fixed $t$, the process $\cB^t_\theta(x)$ is in-fact distributed as the afore-mentioned stationary-horizon, thereby yielding an integrable description of the Busemann process $\cB^t_\theta(x)$
as $x$ and $\theta$ are allowed to vary.

 For $\theta_1<\cdots<\theta_n$, the joint law of $\{\cB^t_{\theta_i}(x)\}_{i=1}^n$ is described in terms of explicit transformations of $n$ independent Brownian motions of diffusivity $2$ and drifts $2\theta_i$. In fact, the above transformation can alternatively be represented as considering last passage times across a line ensemble of independent drifted Brownian motions as described in the heuristics in \cite[Remark 1.3]{Dau21}. For our setting, we will require the joint law of $(\cB^t_0,\cB^t_\theta)$ and we just state the result for this specific setting.

 \begin{proposition}[{\cite[Theorem 5.3, Appendix D]{BSS22}}, {\cite[Lemma 9.2]{SS21}}]
   \label{prop:1}
   Fix $t\in \RR$ and $\theta>0$. There exists a process $\cP^{t,\theta}=(\cP^{t,\theta}_1,\cP^{t,\theta}_2)$ measurable with respect to $\sigma(\cL)$ such that $(\cP^{t,\theta}_1(x),\cP^{t,\theta}_2(x))\stackrel{d}{=}(\sqrt{2}B_1(x),\sqrt{2}B_2(x)+2\theta x)$ as a process in $x$ and further,
   \begin{displaymath}
     (\cB_0^t(x),\cB^t_\theta(x))=\left(\cP_1^{t,\theta}(x),\cP_1^{t,\theta}(x)+\sup_{y\leq x}(\cP_2^{t,\theta}(y)-\cP_1^{t,\theta}(y))-\sup_{y\leq 0}(\cP_2^{t,\theta}(y)-\cP_1^{t,\theta}(y))\right).
   \end{displaymath}
 \end{proposition}

\subsubsection{\textbf{The Brownian absolute continuity of spatial distance profiles}}
The question of establishing the locally Brownian nature of the KPZ fixed point, an object constructed in \cite{MQR21}, has witnessed significant research activity in the past few years. In terms of the directed landscape, this amounts to the statement that for a large class of initial conditions $f$, the function $x\mapsto \sup_y(f(y)+\cL(y,0;x,1))$ is absolutely continuous to Brownian motion on compact intervals, and this was established in \cite{SV21}. For our setting, we will require the following strong Brownian estimates for the so-called `narrow wedge' initial data which we take from \cite{CHH19}.
\begin{proposition}%
  \label{prop:airy}
  There exist constants $C,c$ such that for all $\delta>0$, any $d\in [ -\log^{1/3}\delta^{-1},\log^{1/3}\delta^{-1}]$ and any measurable set $F$ with $\PP(\sqrt{2}B\lvert_{[0,1]}\in F)\geq\delta$,
  \begin{displaymath}
    \PP
    \left(
      (\cL(0,0;d+\cdot,1)-\cL(0,0;d,1))\lvert_{[0,1]}\in F 
    \right)\leq C\PP(\sqrt{2}B\lvert_{[0,1]}\in F)\exp(c|\log \delta^{-1}|^{5/6}).
  \end{displaymath}
\end{proposition}
We note that the narrow wedge profile $\cL(0,0;\cdot,1)$ is in fact the parabolic $\mathrm{Airy}_2$ process, a well-studied determinantal process which originally arose in the theory of random matrices. Though we do not discuss this further, the parabolic $\mathrm{Airy}_2$ process can be embedded as the top curve in a non-intersecting ensemble of curves called the \textit{Airy line ensemble} \cite{CH14} which
has a certain resampling property called the \textit{Brownian Gibbs property}. This resampling property is the source of all Brownian comparison results for the KPZ fixed point including Proposition \ref{prop:airy}.

\subsubsection{\textbf{The spatial Busemann process around a point on the geodesic}}
\label{sec:assump}
  Recall the downward $\theta$-directed semi-infinite geodesic $\Gamma_\theta$ from Section \ref{sec:busem}. The following lemma on the local behaviour of $\cP^{t,\theta}$ around the random point $\Gamma_0(t)$ will be crucial for us.

\begin{lemma}
  \label{cor:ac}
 Define $\cR^{t,\theta}(x)=\cP^{t,\theta}(\Gamma_0(t)+x)$. Then for any fixed $\theta>0$,
 \begin{equation}
   \label{eq:28}
   ((\cR^{t,\theta}_2-\cR^{t,\theta}_1)-(\cR^{t,\theta}_2(0)-\cR^{t,\theta}_1(0)))\lvert_{[-1/4,1/4]}
 \end{equation}
 is absolutely continuous to $(\sqrt{3}B+R)\lvert_{[-1/4,1/4]}$ for all $t>0$. Further, for any fixed $\nu>0,\alpha>1,\kappa\in (0,1/2)$, there exists a constant $C$ such that the following holds for all $\phi\in (0,1), \delta>0$ and $\theta\in (0,\log^{1/2-\kappa}\delta^{-1})$. For any measurable set $H$ with $\PP((\sqrt{3}B+R)\lvert_{[-1/4,1/4]}\in H)=\phi$ and all $t\in [-\alpha,-\alpha^{-1}]$,
  \begin{displaymath}
    \label{eq:130}
    \PP( ((\cR^{t,\theta}_2-\cR^{t,\theta}_1)-(\cR^{t,\theta}_2(0)-\cR^{t,\theta}_1(0)))\lvert_{[-1/4,1/4]}\in H)\leq C(\phi^{1-\nu}\delta^{-2\nu}+\delta^{1-2\nu}).
  \end{displaymath}
\end{lemma}

The above Lemma \ref{cor:ac} is related to the environment seen from geodesics, which has recently witnessed significant activity \cite{MSZ21,DSV22,Bat20,JLS20,BBG21} in many of the different fields encompassing random geometry. The work \cite{DSV22} showed that after doing a KPZ scaling, this environment can be made sense of as a directed landscape with Brownian-Bessel initial data. This question was also investigated \cite{MSZ21} in the prelimiting model of exponential LPP, where explicit formulae for the empirical environment in a $k\times k$ square around points on a geodesic were obtained via a connection to the second class particle in the TASEP. For the present work, we do not require the full environment but rather, a proper understanding of the behaviour of the Busemann process $\cB_\theta^t(\cdot)$ around the point $\Gamma_0(t)$, and this is what Lemma \ref{cor:ac} will allow us to obtain. Note that this is different from considering the local behaviour of $\cB^t_0(\cdot)$ around $\Gamma_0(t)$ which is easier to analyse. %

\subsection{Some properties of geodesics}
\label{sec:trans}
We now introduce some properties of geodesics in the directed landscape which we shall often use. The following result from \cite{GZ22} gives simultaneous control for all finite geodesics. 
\begin{proposition}[Lemma 3.11 in \cite{GZ22}]
  \label{lem:45}
  There is a random number $S$ and constants $C,c>0$ such that the following is true. For $M>0$, we have $\PP(S>M)< Ce^{-cM^{9/4}(\log M)^{-4}}$. Moreover, for any $u=(x,s;y,t)\in \RR_{\uparrow}^4$, and any geodesic $\Gamma_{(x,s)}^{(y,t)}$ and $(s+t)/2\leq r<t$,
  \begin{displaymath}
    \left|
      \Gamma_{(x,s)}^{(y,t)}(r)-\frac{x(t-r)+y(r-s)}{t-s}
    \right|< S(t-r)^{2/3}\log^3
    \left(
      1+\frac{\|u\|}{t-r}
    \right),
  \end{displaymath}
where $\|u\|$ denotes the usual $L^2$ norm.  A similar bound holds when $s<r<(s+t)/2$ by symmetry.
\end{proposition}

Since the above does not quite cover semi-infinite geodesics, we also require the following slightly modified result from %
\cite{RV21}.

\begin{proposition}{{\cite[Theorem 3.12]{RV21}}}
  \label{prop:8}
  There exists a random constant $N$ with $\PP(N\geq t)\leq Ce^{-ct^3}$ for some constants $C,c$ and all $t>0$ such that we have for all $\psi\in \RR$ and $s\geq 0$,
  \begin{displaymath}
    |\Gamma_\psi(-s)-s\psi|\leq Ns^{2/3}(1+\log^{1/3}(|\log s|)).
     \end{displaymath}
\end{proposition}

We shall require a result on the  pre-compactness and coalescence properties of geodesics which we now state.
\begin{proposition}[{\cite[Lemma B.12]{BSS22}, \cite[Lemma 3.3]{DSV22}}]
  \label{prop:fincpt}
  The following holds with probability $1$. For all points $u=(x,s;y,t)\in \RR_\uparrow^4$ and every sequence of points $u_n=(x_n,s_n;y_n,t_n)\in \RR_\uparrow^4$ converging to $u$, and every sequence of geodesics $\Gamma_{(x_n,s_n)}^{(y_n,t_n)}$, there exists a subsequence $\{n_i\}$ along which the geodesics $\Gamma_{(x_n,s_n)}^{(y_n,t_n)}$ converge uniformly to a geodesic $\Gamma_{(x,s)}^{(y,t)}$. Further, if the geodesic $\gamma_{(x,s)}^{(y,t)}$ is unique or the geodesics $\gamma_{(x_n,s_n)}^{(y_n,t_n)}$ are unique for every $n$, then for every $\delta>0$, we have $\gamma_{(x_{n_i},s_{n_i})}^{(y_{n_i},t_{n_i})}(r)=\gamma_{(x,s)}^{(y,t)}(r)$ for all $r\in [s+\delta,t-\delta]$ and all $i$ large enough.
\end{proposition}
In fact, by using the above, it is not difficult (see e.g.\ \cite[Lemma 14]{Bha23}) to obtain a version of the above result for semi-infinite geodesics, and we now state such a version.
  \begin{proposition}
    \label{prop:cpt}
 Fix $\psi\in \RR$. The following holds with probability $1$. For all $q=(y,t)$ and any sequence of rational points $q_n\rightarrow q$, every sequence of geodesics $\Gamma_{\psi}^{q_n}$, there exists a subsequence $\{n_i\}$ such that $\Gamma_{\psi}^{q_{n_i}}$ converges to a geodesic $\Gamma_\psi^p$ in the locally uniform topology as $i\rightarrow \infty$. In fact, for every $\delta>0$, we have $\Gamma_{\psi}^{q_{n_i}}(s)= \Gamma_\psi^p(s)$ for all $s\leq t-\delta$ and all large enough $i$.
\end{proposition}

\subsection{The difference profile}
\label{ss:diff}
The primary technique used in the literature to obtain dimension lower bounds for atypical stars goes via an object called the difference profile (introduced in \cite{BGH21}) which we shall rely on as well. The usually considered difference profile is the function $\fD^{\mathrm{un}}(x)=\cL( 1,0,x,1)-\cL( 0,0, x,1)$ from \eqref{eq:128}. However, in our setting, we will be working with semi-infinite geodesics and will thus work with the profile given by the difference of two spatial Busemann processes instead. We now set up notation and collect the main results we shall need pertaining to difference profiles.

We will in fact use two distinct types of difference profiles. The former, which we denote as $\cD_\theta^t(\cdot)$, is a spatial difference profile and is defined for every $\theta\in \RR, t\in (-\infty,0)$ as
\begin{equation}
  \label{eq:40}
  \cD_\theta^t(x)=\cB_\theta^t(\Gamma_0(t)+ x)-\cB_0^t(\Gamma_0(t)+x).
\end{equation}
Though the definition of $\cD_\theta^t$ above looks very different from $\fD^\mathrm{un}$ at first sight, it is not so. Instead of measuring distances to the points $(1,0)$ and $(0,0)$, \eqref{eq:40} captures such differences to the ``points at $\infty$'' in the directions $0$ and $\theta$. Further, $\cD_\theta^t$ is rooted around the random point $\Gamma_0(t)$ in order to suit our setting of looking at profiles in the vicinity of the geodesic.
 By using the directed nature of the model, it was shown in \cite{BGH21} that each $\fD^{\mathrm{un}}$ is actually an increasing function \cite{BGH21}. Further, \cite{BGH21,BGH19} investigated the set of non-constancy points of $\fD^{\mathrm{un}}$ and computed their Hausdorff dimension. Though the above works consider $\fD^{\mathrm{un}}$, the same proofs go through to yield corresponding results for $\cD_\theta^t$. To avoid repetition, we do not provide proofs but just state the corresponding results.
\begin{proposition}[\cite{BGH21,BGH19}]
  \label{prop:7}
  For each fixed $\theta\in \RR,t\in (-\infty,0)$, the function $\cD_\theta^t$ is increasing. Let $\mu_{\cD_\theta^t}$ denote the Lebesgue-Stieltjes measure associated to $\cD_\theta^t$. Then the set $\cX^t_\theta$ consisting of points $x$ admitting two almost disjoint downward semi-infinite geodesics $\Gamma^{(\Gamma(t)+x,t)}_0,\Gamma^{(\Gamma(t)+x,t)}_\theta$
  has the property $\cX_\theta^t=\supp \mu_{\cD_\theta^t}$ along with $\dim \cX_\theta^t =1/2$ almost surely. 
\end{proposition}

 The other difference profile that we use is the geodesic difference profile, which we denote simply as $\cD_\theta$ for any $\theta>0$ and define as in \eqref{eq:31} by
\begin{equation}
  \label{eq:39}
   \cD_\theta(t)=\cB_\theta((\Gamma_0(t),t))-\cB_0((\Gamma_0(t),t)).
\end{equation}
for $t\in (-\infty,0]$. We emphasize that $\cD_\theta$ is only well-defined as a function modulo a global additive constant. The following lemma shows that just like spatial difference profiles, the geodesic difference profile is monotonic. 
\begin{lemma}
  \label{lem:10}
 Fix $\theta>0$. Almost surely, $\cD_\theta(t)$ is monotonically increasing for $t\in (-\infty,0]$.
\end{lemma}
\begin{proof}
  It suffices to show that almost surely, for all $s<t\leq 0$, $\cD_\theta(t)\geq \cD_\theta(s)$ and this is equivalent to showing the inequality $\cB_\theta((\Gamma_0(t),t))\geq\cB_\theta((\Gamma_0(s),s))+\cL(\Gamma_0(s),s;\Gamma_0(t),t)$ which is true since $\cB_\theta((\Gamma_0(t),t))=\max_{x\in \RR}(\cB_\theta((x,s))+\cL(x,s;\Gamma_0(t),t))$.
\end{proof}

Though the above lemma is not actually used in this paper, it gives valuable intuition about why the $1/3$ lower bound in Theorem \ref{thm:2} coming from the difference profile should be tight; see Remark \ref{rem:int}. By using coalescence in a similar manner as the proof of the corresponding statement for $\cD_\theta^t$, it is not difficult to show that any point $t$ of non-constancy of $\cD_\theta$  (see Figure \ref{fig:dis}) must be such that $(\Gamma_0(t),t)$ admits a pair of almost disjoint downward semi-infinite geodesics directed along the angles $0$ and $\theta$. %
\begin{lemma}
  \label{lem:11}
  For a fixed $\theta\neq 0$, let $\mu_{\cD_\theta}$, a random measure on $(0,1)$, be the Lebesgue-Stieltjes measure of $\cD_\theta$. Then almost surely, $\supp\mu_{\cD_\theta}\subseteq \mult_{\theta}$. In other words, almost surely, every $s\in \supp \mu_{\cD_\theta}$ admits a pair of almost disjoint geodesics $\Gamma_{0}^{(\Gamma_0(s),s)},\Gamma_{\theta}^{(\Gamma_0(s),s)}$.
\end{lemma}
\begin{proof}
  Since these ideas have already appeared in the literature, we will be brief. To begin, by the symmetries of the directed landscape (see \cite[Proposition 1.23]{DV21}), it suffices to consider $\theta>0$. Further, it suffices to show that a.s.\ simultaneously for all points $p=(x,s)$ which do not have two almost disjoint geodesics $\Gamma^p_0$ and $\Gamma_{\theta}^p$, the up to global additive constants defined map $z\mapsto \cB_\theta(z)-\cB_0(z)$ for $z\in \RR^2$ is constant in a two dimensional neighbourhood of $p$. Let $\gamma^{\mathrm{L}}$ denote the leftmost \cite[Theorem 3.14]{RV21} choice of $\Gamma^p_{0}$ and $\gamma^{\mathrm{R}}$ denote the rightmost choice of $\Gamma_{\theta}^p$. For $n\in \NN$, choose rational points $p^{\mathrm{L}}_n=(x^{\mathrm{L}}_n,s_n),p^{\mathrm{R}}_n=(x^{\mathrm{R}}_n,s_n)$ such that $|s_n-(s+n^{-1})|\leq n^{-1}/2$ and $|x^{\mathrm{L}}_n-(x-n^{-1/2})|,|x^{\mathrm{R}}_n-(x+n^{-1/2})|\leq n^{-1/2}/2$ and define $\gamma^{\mathrm{L}}=\Gamma_{0}^{p_n^{\mathrm{L}}}$ and $\gamma^{\mathrm{R}}=\Gamma_{\theta}^{p_n^{\mathrm{R}}}$. By using Proposition \ref{prop:8}, it follows that for all large $n$, $\gamma_n^{\mathrm{L}}(s)<\gamma^{\mathrm{L}}(s)=\gamma^{\mathrm{R}}(s)<\gamma_n^{\mathrm{R}}(s)$ and thus $\gamma_n^{\mathrm{L}}(t)\leq\gamma^{\mathrm{L}}(t)=\gamma^{\mathrm{R}}(t)\leq\gamma_n^{\mathrm{R}}(t)$ for all $t\in [0,s]$. By compactness properties of geodesics (Proposition \ref{prop:fincpt}), $\gamma^{\mathrm{L}}_n\lvert_{(-\infty,s]},\gamma^{\mathrm{R}}_n\lvert_{(-\infty,s]}$ converge locally uniformly to $\gamma^{\mathrm{L}},\gamma^{\mathrm{R}}$ respectively, and further, for all $n$ large enough, $\gamma^{\mathrm{L}}_n\cap\gamma^{\mathrm{L}}\cap \gamma^{\mathrm{R}}\cap\gamma^{\mathrm{R}}_n\neq \emptyset$. Fix such a value of $n$ and let $(\beta,\tau)$ denote the point in $\gamma^{\mathrm{L}}_n\cap\gamma^{\mathrm{L}}\cap \gamma^{\mathrm{R}}\cap\gamma^{\mathrm{R}}_n$ with the largest time coordinate. Now the profile $z\mapsto  \cB_\theta(z)-\cB_0(z)$ can be seen (see \cite[Lemma 3.4]{BGH21}) to be constant for all $z\in \{(y,t):t\in(\tau,s_n),s\in(\gamma^{\mathrm{L}}_n(t),\gamma^{\mathrm{R}}_n(t))\}$ which is a two dimensional neighbourhood of $p$. This completes the proof.
\end{proof}
Finally, we show that the difference profile $\cD_\theta$ is almost surely locally $1/3-$ H\"older continuous and this will be important to us later for obtaining dimension lower bounds.
\begin{lemma}
  \label{lem:48}
 Fix $\theta>0$. For any $\varepsilon>0$, $\cD_\theta$ is almost surely locally H\"older continuous with exponent $1/3-\varepsilon$.
\end{lemma}
\begin{proof}
  By using that $\sup_{s\in[2t,0]}|\Gamma^0(s)|<\infty$ for each $t\in (-\infty,0)$ along with the uniform transversal fluctuation estimate Proposition \ref{lem:45}, we obtain that $\Gamma_0\lvert_{[0,2t]}$ is a.s.\ $2/3-\varepsilon$ H\"older continuous for any $\varepsilon>0$. As a result, we know that $|\Gamma_0(t+\delta)-\Gamma_0(t)|\leq \delta^{2/3-\varepsilon}$ for all $\delta$ small enough. Now, by using the $1/3-$ temporal and the $1/2-$ spatial H\"older regularity of the directed landscape (see \cite[Proposition 10.5]{DOV18}), it is not difficult to obtain that $|\cD_\theta(t+\delta)-\cD_\theta(t)|\leq \delta^{1/3-\varepsilon}$ for all small enough $\delta$ depending on $t$. This completes the proof.
\end{proof}

\begin{figure}
  \centering
  \captionsetup[subfigure]{labelformat=empty}
  \setbox9=\hbox{\includegraphics[width=0.4\textwidth]{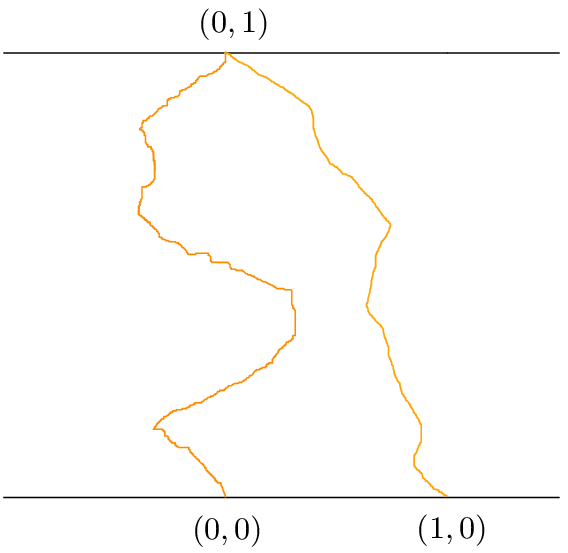}}
  \subcaptionbox{}{\raisebox{\dimexpr\ht9-\height}{\includegraphics[width=0.4\textwidth]{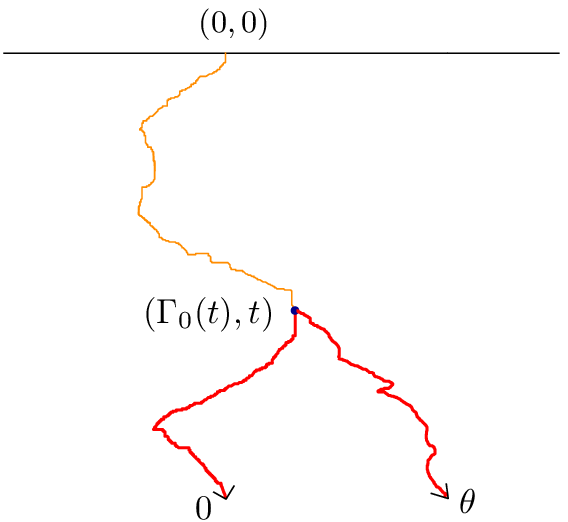}}}
  \hfill
  \subcaptionbox{}{\includegraphics[width=0.4\textwidth]{inc_spatial_diff}}
  \caption{In the panel to the left, $(\Gamma_0(t),t)$ is a non-constancy point of the geodesic difference profile $\cD_\theta$. In contrast, in the panel on the right, $(0,1)$ is a non-constancy point for the spatial difference profile $\fD^{\mathrm{un}}$ from \eqref{eq:128}.}
  \label{fig:dis}
\end{figure}

\subsection{Tails for exit times of Brownian motion from cones}
\label{ss:cones}
An interesting aspect of the proof of the upper bound in Theorem \ref{thm:2} is that it finally reduces to computing the smallest eigenvalue of the Laplace-Beltrami operator on a $4$-dimensional spherical cap of angle $\pi/3$. In this final subsection, we gather some results related to this part of the proof.

Let $D$ denote a proper open connected subset of $S^{d-1}$, the $(d-1)$-dimensional unit sphere regarded as a subset of $\RR^d$. Define the generalized cone $\cC^D$ generated by $D$ as the set of all rays emanating from $\mathbf{0}\in\RR^d$ and passing through $D$. Let $\Delta_{S^{d-1}}$ denote the Laplace-Beltrami operator on $S^{d-1}$ and assume that $D$ is regular
for the Dirichlet problem with respect to $\Delta_{S^{d-1}}$. Under this assumption, there exists a sequence of eigenvalues $0=\lambda^D_0<\lambda^D_1<\lambda^D_2<\dots$ and orthogonal eigenspaces $\cH_i\subseteq L^2(D)$ such that all $f\in \mathcal{H}_i$ satisfy $f\lvert_{\partial D}=0$, $\Delta_{S^{d-1}}f=-\lambda^D_i f$ and the decomposition
 $ L^2(D)=\bigoplus_{i=0}^{\infty}\cH_i$ holds. The dependence of the tail probabilities of the exit time of Brownian motion from a cone $\cC^D$ on the Laplace spectrum of $D$ were studied in \cite{Bur77,DeB87,BS97} and we now state such a result.
\begin{proposition}[Corollary 1 in \cite{BS97}]
  \label{prop:4}
  Let $D$ be a proper regular open connected subset of $S^{d-1}$ and let $x$ be an interior point of $\cC^D$. Let $\tau^D_x$ denote the time it takes for a $d$-dimensional Brownian motion started from $x$ to exit the cone $\cC^D$. Then $\tau^D_x$ has power law tails with the tail exponent $\frac{\sqrt{\lambda_1^D+(d/2-1)^2}-(d/2-1)}{2}$.
\end{proposition}
We will be concerned with the case when $\cC^D$ is a right circular cone, or alternatively, when the set $D\subseteq S^{d-1}$ is a spherical cap of some angle $\nu_0$. In this case, the work \cite{BCG83} has a description of the eigenvalue $\lambda^D_1$ in terms of the eigenvalues of a certain PDE. Curiously, in the special case of $d=4$, this PDE has a particularly simple form
which enables an exact calculation of $\lambda_1^D$.
\begin{proposition}[Lemma 3.2, Remark 3.1 in \cite{BCG83}]
  \label{prop:5}
  Let $D\subseteq S^3\subseteq \RR^4$ be a spherical cap of angle $\nu_0$. Then $\lambda^1_D$ is the smallest eigenvalue of the equation
  \begin{displaymath}
    \frac{d^2u}{d\nu^2}+(\mu+1)u=0,
  \end{displaymath}
  which in fact can be explicitly calculated as $\mu(\nu_0)=\frac{\pi^2}{\nu_0^2}-1$ with the corresponding eigenfunction $u(\nu)=\sin((\sqrt{\mu(\nu_0)+1})\nu)$. In particular, if $D\subseteq S^3$ is a spherical cap with angle $\nu_0=\pi/3$, then $\lambda_1^D=\mu(\pi/3)=8$.
\end{proposition}

\begin{corollary}
  \label{cor:1}
  Let $D\subseteq S^3$ be a spherical cap of angle $\pi/3$ and let $x\in \cC^D$ be an interior point. Then $\PP(\tau_x^D\geq t)=\Theta(t^{-1})$.
  Also, if $R$ is a standard Bessel-$3$ Process and $B$ is an independent standard Brownian motion, then $\PP(\max_{x\in [0,1]}(\sqrt{3}B(x)+R(x))\leq \delta )=\Theta(\delta^2)$.
\end{corollary}
\begin{proof}
  The statement regarding the tail exponent of $\tau_x^D$ follows by using the expression from Proposition \ref{prop:4} along with Proposition \ref{prop:5}. For the statement regarding $\sqrt{3}B+R$, we reduce it to the first statement. First note that by Brownian scaling, it suffices to show that for large $t$,
  \begin{equation}
    \label{eq:37}
    \PP
    \left(
      \max_{x\in [0,t]}(\sqrt{3}B(x)+R(x))\leq  1
    \right)=\Theta(t^{-1}).
  \end{equation}
  Recall the Brownian motions $\{B_i\}_{i\geq 0}$ from the notational comments just before Section \ref{sec:ideas} which we take to be independent of $B$ here.  Note that we can decompose $R$ as $\sqrt{B_1^2+B_2^2+B_3^2}$.
    With this notation, \eqref{eq:37} can be reduced to establishing $\Theta(t^{-1})$ tails for the first time $S$ when the process $\sqrt{3} (B-1/\sqrt{3}) + \sqrt{B_1^2+B_2^2+B_3^2})$ hits the value $0$. Now we observe that $S$ is simply the exit time of the Brownian motion $(-1/\sqrt{3}+B,B_1,B_2,B_3)$ from the cone
  \begin{equation}
    \label{eq:38}
    \left\{(z,x_1,x_2,x_3)\in \RR^4\colon \sqrt{3}z\leq -\sqrt{x_1^2+x_2^2+x_3^2}\right\},
  \end{equation}
  and we note that the above cone can be realized as $\cC^D$ for a spherical cap $D\subseteq S^3\subseteq \RR^4$ of angle $\pi/3$. Corollary \ref{cor:1} now immediately yields the desired $\Theta(t^{-1})$ estimate on the exit time.
\end{proof}

\begin{lemma}
  There exists a constant $C_1$ such that for all $\delta>0$,
  \label{lem:32}
  \begin{displaymath}
    \PP
    \left(
      \argmax_{x\in [-\delta,1]}(\sqrt {3} B(x)+R(x))\in [-\delta,\delta]
    \right)\leq \PP
    \left(
      \max_{x\in [0,1]}(\sqrt{3}B(x)+R(x))\leq \sqrt{\delta\log \delta^{-1}}
    \right)+C_1\delta.
  \end{displaymath}
\end{lemma}

\begin{proof}
  Recall that for a Brownian motion $B$,
  \begin{align}
    \label{eq:64}
    &\PP
    \left(
      \max_{x\in [-\delta,\delta]}B(x)\geq \sqrt{2\delta \log\delta^{-1}}
    \right)\nonumber\\&\leq 2\PP
    \left(
      \max_{x\in [0,\delta]}B(x)\geq \sqrt{2\delta \log\delta^{-1}}
    \right)=2\PP
    \left(
      B(\delta)\geq \sqrt{2\delta\log \delta^{-1}}\right)\leq C\delta.
  \end{align}
  Writing $R=\sqrt{B_1^2+B_2^2+B_3^2}$ and using the above, we obtain
  \begin{equation}
    \label{eq:65}
    \PP
    \left(
      \max_{x\in[-\delta,\delta]}(\sqrt{3}B(x)+R(x))\geq \frac{\sqrt{2\delta\log \delta^{-1}}}{2\sqrt{3}}
    \right)\leq 4 \PP
    \left(
      \max_{x\in [-\delta,\delta]}B(x)\geq \sqrt{2\delta \log\delta^{-1}}
    \right)\leq C_1\delta.
  \end{equation}
  Thus we can now write
  \begin{align}
    \label{eq:66}
    &
      \PP
    \left(
      \argmax_{x\in [-\delta,1]}(\sqrt {3} B(x)+R(x))\in [-\delta,\delta]
    \right)\nonumber\\
    &\leq \PP\left(
      \max_{x\in[-\delta,\delta]}(\sqrt{3}B(x)+R(x))\geq \frac{\sqrt{\delta\log \delta^{-1}}}{\sqrt{6}}
    \right)+\PP
    \left(
      \max_{x\in [-\delta,1]}(\sqrt{3}B(x)+R(x))\leq \frac{\sqrt{\delta\log \delta^{-1}}}{\sqrt{6}}
      \right)\nonumber\\
    &\leq C_1\delta + \PP
    \left(
      \max_{x\in [0,1]}(\sqrt{3}B(x)+R(x))\leq \sqrt{\delta\log \delta^{-1}}
    \right),
  \end{align}
  where in the last line, we used \eqref{eq:65}, along with the facts $1/(\sqrt{6})<1$ and $[0,1]\subseteq [-\delta,1]$.
\end{proof}

\begin{corollary}
  \label{cor:max}
 For $\delta>0$, we have $\PP
  \left(
    \argmax_{x\in [-\delta,1]} (\sqrt{3}B(x)+R(x))\in [-\delta,\delta]
  \right)=O(\delta\log\delta^{-1})$.
\end{corollary}
\begin{proof}
  The proof follows by invoking Lemma \ref{lem:32} and applying the second statement of Corollary \ref{cor:1} with $\delta$ replaced by $\sqrt{\delta\log \delta^{-1}}$.
\end{proof}

\section{Proof of Theorem \ref{thm:2} and Theorem \ref{thm:1}}
In this section, we provide the proof of the main results. As discussed in Section \ref{sec:ideas}, the dimension lower bound is simpler, while the upper bound needs new ideas. We now %
obtaining a lower bound on the dimension of $\mult_{\theta}$, a subset of the set of $t\in (-\infty,0)$ for which $(\Gamma_0(t),t)$ is an atypical star.

\begin{lemma}
  \label{lem:200}
  For each $\theta\neq 0$, we have $\dim \mult_\theta \geq 1/3$. Further, almost surely, for any $\delta>0$, $\dim\mult_\theta\cap(-\delta,0)\geq 1/3$ for all $\theta$ large enough depending on $\delta$.
\end{lemma}
\begin{proof}

  By the symmetries of the directed landscape (see \cite[Proposition 1.23]{DV21}), it suffices to consider $\theta>0$.
  We use a H\"older continuity argument using the geodesic difference profile as indicated in Section \ref{sec:ideas}; we do not provide the complete details of this part of the argument since these ideas have been used in the literature earlier \cite{BGH21}. Using the $1/3-$ local H\"older continuity of $\cD_\theta$ from Lemma \ref{lem:48} along with the inclusion $\supp\mu_{\cD_\theta}\subseteq \mult_{\theta}$ from Lemma \ref{lem:11}, the first statement of the lemma can be reduced to showing that for all small enough $t$,
  \begin{equation}
    \label{eq:125}
\cD_\theta(0)>\cD_\theta(t),
\end{equation}
 and this is equivalent to showing $\cB_\theta((0,0))>\cB_\theta((\Gamma_0(t),t))+\cL(\Gamma_0(t),t;0,0)$. Thus by the a.s.\ uniqueness of $\Gamma_0$, it suffices to show that for all negative enough $t$, $(\Gamma_\theta(t),t)\notin \Gamma_0$ and this is obvious since $\lim_{s\rightarrow \infty}\Gamma_\theta(-s)/s=\theta>0=\lim_{s\rightarrow \infty}\Gamma_0(-s)/s$.

 For the second part of the lemma, it suffices to show that for any fixed $\delta>0$, $\cD_\theta(0)>\cD_\theta(-\delta)$ holds for all large enough $\theta$ and this can be reduced as above to showing that a.s.\ for all large $\theta$, $(\Gamma_\theta(-\delta),-\delta)\notin \Gamma_0$ and this can be seen by sending $\psi\rightarrow \infty$ in Proposition \ref{lem:45}.

\end{proof}

The rest of this section focuses on dimension upper bounds. The aim is to prove the following crucial lemma which will be used in conjunction with Lemma \ref{lem:200} at the end of the section to complete the proof. 
\begin{lemma}
  \label{lem:5}
 For any fixed $\alpha>1,\theta>0,\beta>0$, there exists a constant $C$ such that for all $\varepsilon\in (0,1)$ and for all $i\in [\![-\alpha \varepsilon^{-1},-\alpha^{-1} \varepsilon^{-1}]\!]$,
  \begin{equation}
    \label{eq:8}
    \PP
    \left(
      \mult_{\theta}\cap [i\varepsilon,(i+1)\varepsilon]\neq \emptyset
    \right)\leq C\varepsilon^{2/3-\beta}.
  \end{equation}
\end{lemma}
As discussed in Section \ref{sec:ideas}, we need to show that on the event $\{\mult_\theta\cap [i\varepsilon,(i+1)\varepsilon]\neq \emptyset\}$, the occurrence of $\dis_\theta\{J_{i,\varepsilon^{2/3}}\}$ is very likely. The following rather technical lemma on the transversal fluctuations of geodesics will help us achieve this.
\begin{lemma}
  \label{lem:4}
  Fix $\alpha>1$ and $\theta>0$. For any $\varepsilon>0$ and any interval $I=(s,s+\varepsilon)\subset [-\alpha,-\alpha^{-1}]$, consider the event $\cM_\theta[I]$ given by
  \begin{displaymath}
 \cM_\theta[I]= \left\{
    |\Gamma^{(\Gamma_0(t),t)}_\theta(s)-\Gamma_0(s)|> \varepsilon^{2/3}\log^4 \varepsilon^{-2/3} \text{ for some } t\in I \text{ and some geodesic } \Gamma^{(\Gamma_0(t),t)}_\theta
  \right\}.
\end{displaymath}
Then $\PP(\cM_\theta[I])\rightarrow 0$ superpolynomially as $\varepsilon\rightarrow 0$.
\end{lemma}
\begin{proof}
  Define $T=\max_{s\in [-\alpha,0]}(\max(|\Gamma_0(s)|,|\Gamma_\theta(s))|)$ %
 and note that by Proposition \ref{prop:8}, the bound $\PP(T\geq M)\leq Ce^{-cM^3}$ holds. 
   Now define the event $E=\{T\geq \log \varepsilon^{-1}\}$ and note that $\PP(E)$ goes to zero superpolynomially in $\varepsilon$. %
 Using the above, we have
  \begin{align}
    \label{eq:117}
    \PP
    \left(
    \cM_\theta[I]
    \right)&\leq \PP(E)+\PP\left(E\cap \{  |\Gamma^{(\Gamma_0(t),t)}_{\theta}(s)-\Gamma^{(\Gamma_0(t),t)}_{0}(s)|> \varepsilon^{2/3}\log^4 \varepsilon^{-2/3} \text{ for some } t\in I\}\right).
  \end{align}
Note that there might be multiple geodesics $\Gamma_\theta^{(\Gamma_0(t),t)}$ and the interpretation of the probability in the above expression is that the event in question holds for some choice of geodesics; this interpretation will be used for all the events and probabilities used in this proof. Now, we already know that $\PP(E)$ decays superpolynomially in $\varepsilon$ and it suffices to show the superpolynomial decay of the latter term. Denote the event considered in the latter term in \eqref{eq:117} by $F$.

   Note that $\Gamma_0(s)=\Gamma^{(\Gamma_0(t),t)}_{0}(s)$ for all $s<t$. We now define $x^t_0=\Gamma_0(-\alpha)=\Gamma_0^{(\Gamma_0(t),t)}(-\alpha)$ and $x_\theta^t=\Gamma^{(\Gamma_0(t),t)}_{\theta}(-\alpha)$. %
 By the a.s.\ uniqueness of the geodesics $\Gamma_0,\Gamma_\theta$, it is not difficult to observe that any geodesic $\Gamma_\theta^{(\Gamma_0(t),t)}$ must lie in between $\Gamma_0$ and $\Gamma_\theta$ and thus on the event $E$, 
  \begin{equation}
    \label{eq:20}
    |x_0^t|,|x_\theta^t| \leq \log \epsilon^{-1}.
  \end{equation}
   By observing that $\Gamma_0^{(\Gamma_0(t),t)}(s)=\Gamma_{(x^t_0,-\alpha)}^{(\Gamma_0(t),t)}(s)$ and $\Gamma_\theta^{(\Gamma_0(t),t)}(s)=\Gamma_{(x^t_\theta,-\alpha)}^{(\Gamma_0(t),t)}(s)$ %
   for all $s\in I$, we can write
  \begin{equation}
    \label{eq:17}
    F=
  E\cap \left\{  |\Gamma_{(x^t_\theta,-\alpha)}^{(\Gamma_0(t),t)}(s)-\Gamma_{(x^t_0,-\alpha)}^{(\Gamma_0(t),t)}(s)|> \varepsilon^{2/3}\log^4 \varepsilon^{-2/3} \text{ for some } t\in I\right\}.
  \end{equation}
  The remainder of the proof is an application of Proposition \ref{lem:45} to upper bound $\PP(F)$. First, define $m_0=(t+\alpha)^{-1}(x^t_0(t-s)+\Gamma_0(t)(s+\alpha))$ and $m_\theta=(t+\alpha)^{-1}(x^t_\theta(t-s)+\Gamma_0(t)(s+\alpha))$ and note that these are the centering terms which appear if Proposition \ref{lem:45} is applied with the geodesics $\Gamma_{(x^t_0,-\alpha)}^{(\Gamma_0(t),t)}(s)$ and $\Gamma_{(x^t_\theta,-\alpha)}^{(\Gamma_0(t),t)}(s)$.

 By \eqref{eq:20}, on the event $E$, we have $|m_0-m_\theta|=|(x^t_\theta-x^t_0)(t-s)(t+\alpha)^{-1}|\leq C_2\varepsilon\log \varepsilon^{-1}$ for some constant $C_2$ depending on $\alpha$. %
   Thus we can write
  \begin{align}
    \label{eq:118}
    \PP(F)&\leq \PP(E\cap\{\left|\Gamma_{(x^t_0,-\alpha)}^{(\Gamma_0(t),t)}(s)-m_0\right|>\varepsilon^{2/3}(\log^4 \varepsilon^{-2/3})/2-C_2\varepsilon\log \varepsilon^{-1}\})\nonumber\\
    &+\PP(E\cap \{\left|\Gamma_{(x^t_\theta,-\alpha)}^{(\Gamma_0(t),t)}(s)-m_\theta\right|>\varepsilon^{2/3}(\log^4 \varepsilon^{-2/3})/2-C_2\varepsilon\log \varepsilon^{-1}\}).
  \end{align}
  We now just show that the second term above has superpolynomial decay in $\varepsilon$; the proof for the first term is the same except for applying Proposition \ref{lem:45} with a different value for $u$.
  Looking at the second term, we first note that $\varepsilon^{2/3}(\log^4\varepsilon^{-2/3})/2-C_2\varepsilon\log \varepsilon^{-1}\geq \varepsilon^{2/3}(\log^4 \varepsilon^{-2/3})/4$ for all $\varepsilon$ small and thus we only need to show the superpolynomial decay of
  \begin{equation}
    \label{eq:119}
    \PP(E\cap \{\left|\Gamma_{(x^t_\theta,-\alpha)}^{(\Gamma_0(t),t)}(s)-m_\theta\right|>\varepsilon^{2/3}(\log^4 \varepsilon^{-2/3})/4\}).
  \end{equation}
If we define $u=(x_\theta^t,-\alpha;\Gamma^0_t,t)$, then by \eqref{eq:20}, on the event $E$, %
  $\|u\|\leq C_3\log \varepsilon^{-1}$ holds for some constant $C_3$. Now by using that $(t-s)^{2/3}\leq \varepsilon^{2/3}$, we have on the event $E$, $\log^3(1+\|u\|/(t-s))\leq C_4[\log(\varepsilon^{-2/3}\log \varepsilon^{-1})]^3\leq C_5\log^3\varepsilon^{-1}$. Thus by applying Proposition \ref{lem:4} with the above $u$ and $r=s$, we obtain
  \begin{equation}
    \label{eq:120}
    \PP(E\cap \{\left|\Gamma_{(x^t_\theta,-\alpha)}^{(\Gamma_0(t),t)}(s)-m_\theta\right|>\varepsilon^{2/3}(\log^4 \varepsilon^{-2/3})/4\})\leq \PP(S\geq C_5^{-1}\log\varepsilon^{-1}),
  \end{equation}
and the latter decays superpolynomially in $\varepsilon$ by the tails of $S$ from Proposition \ref{lem:45}. This shows the superpolynomial decay of $\PP(F)$ and completes the proof.
\end{proof}

The following lemma (c.f.\ \cite[Proof of Proposition 6.10]{GZ22}) yields a precise form of the intuitive statement mentioned just before Lemma \ref{lem:4}.
\begin{figure}
  \centering
  \captionsetup[subfigure]{labelformat=empty}
  \subcaptionbox{}{\includegraphics[width=0.4\textwidth]{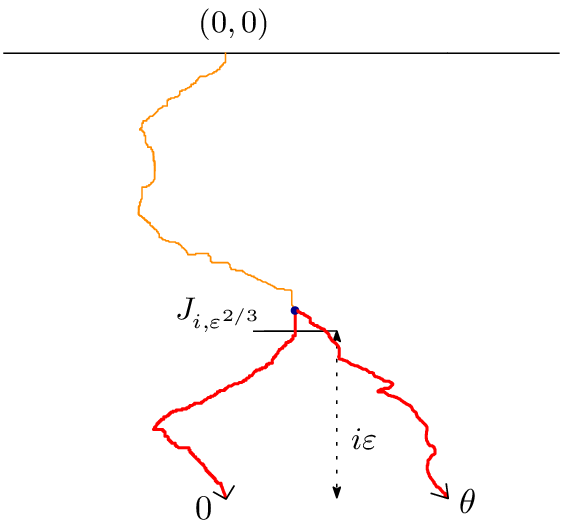}}
  \hfill
  \subcaptionbox{}{\includegraphics[width=0.4\textwidth]{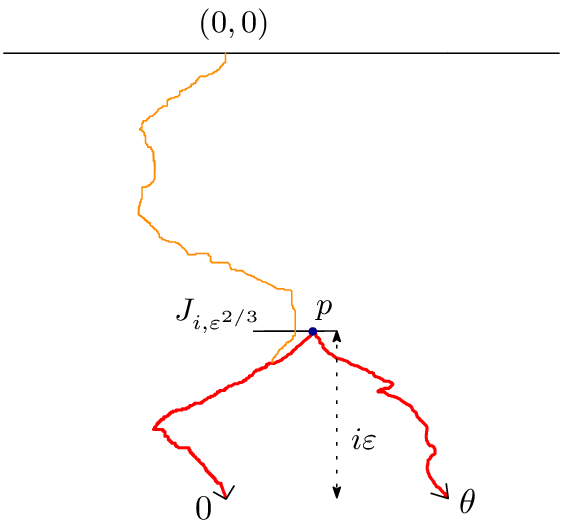}}
  \caption{In the left panel, $\mult_\theta\cap [i\varepsilon,(i+1)\varepsilon]$ occurs and $\cM_\theta[J_{i,\varepsilon^{2/3}}]$ does not occur, and thus there exist two disjoint downward $0$ and $\theta$-directed semi-infinite geodesics starting at points in $J_{i,\varepsilon^{2/3}}$. In the panel to the right, we see that the disjoint geodesics on the left imply the existence of a point $p\in J_{i,\varepsilon^{2/3}}$ which has a pair of almost disjoint geodesics $\Gamma^p_0, \Gamma^p_\theta$, and thereby the occurrence of $\dis_\theta\{J_{i,\delta}\}$.}
  \label{fig:coal_trans}
\end{figure}
\begin{lemma}
  \label{lem:2}
    For any fixed $\theta>0$ and $\alpha>1$, we have for all $i\in [\![-\alpha \varepsilon^{-1},-\alpha^{-1} \varepsilon^{-1}]\!]$ and $\varepsilon\in (0,1)$,
   \begin{displaymath}
    \{ \mult_{\theta}\cap [i\varepsilon,(i+1)\varepsilon]\neq \phi\}\subseteq  \dis_\theta\{J_{i,\varepsilon^{2/3}}\}\cup \cM_\theta[J_{i,\varepsilon^{2/3}}].
  \end{displaymath}
\end{lemma}
\begin{proof}
  Assuming that the set on the left hand side is non-empty, we pick any $t\in \mult_{\theta}\cap  [i\varepsilon,(i+1)\varepsilon]$. By the definition of $\mult_{\theta}$, there exists a downward $\theta$-directed semi-infinite geodesic $\widetilde{\Gamma}$ from  $(\Gamma_0(t),t)$ such that $\Gamma_0\lvert_{(-\infty,t)}$ and $\widetilde{\Gamma}$ are almost disjoint. Now in case $\cM_\theta[J_{i,\varepsilon^{2/3}}]$ does hold, we are already done. If $\cM_\theta[J_{i,\varepsilon^{2/3}}]$ does not hold, then $|\widetilde{\Gamma}(i\varepsilon)-\Gamma_0(i\varepsilon)|\leq \varepsilon^{2/3}\log^4\varepsilon^{-2/3}$ and thus $(\widetilde{\Gamma}(i\varepsilon),i\varepsilon)\in J_{i,\varepsilon^{2/3}}$. Thus $\Gamma_0\lvert_{(-\infty,i\varepsilon]},\widetilde{\Gamma}\lvert_{(-\infty,i\varepsilon]}$ are two disjoint $0$ and $\theta$-directed downward semi-infinite geodesics starting in the interval $J_{i,\varepsilon^{2/3}}$.

    Now, we claim that the above implies that there must exist a $p\in J_{i,\varepsilon^{2/3}}$ admitting a pair of almost disjoint geodesics $\Gamma^p_0$, $\Gamma^p_\theta$. Such a construction has appeared earlier in the literature (see Proposition 3.5 in the arXiv version of \cite{BGH21}), and thus we shall be brief. Define
  \begin{equation}
    \label{eq:1}
    x_*=\sup\{x: x\in [\Gamma_0(i\varepsilon),\widetilde \Gamma(i\varepsilon)], \Gamma_0^{(x,i\varepsilon)}(s)< \widetilde \Gamma(s) \textrm{ for all } s\leq i\varepsilon \textrm{ and all geodesics } \Gamma_0^{(x,i\varepsilon)}\},
  \end{equation}
  and by using the pre-compactness properties of geodesics (Proposition \ref{prop:cpt}) along with the uniqueness of $\Gamma_0$, it is not difficult to see that $x_*>\Gamma_0(i\varepsilon)$. We now argue under the assumption that $x_*<\widetilde \Gamma(i\varepsilon)$, and the boundary case can be handled similarly. The goal now is to show that it suffices to take $p=(x_*,i\varepsilon)$. As a consequence of planarity (see the proof of \cite[Proposition 3.5]{BGH21}), it can be shown that for every $x\in [\Gamma_0(i\varepsilon),x_*)$, $y\in (x_*,\widetilde \Gamma(i\varepsilon)]$, the right-most geodesics $\Gamma_{0}^{(x,i\varepsilon)}$ and $\Gamma_{\theta}^{(y,i\varepsilon)}$ are disjoint. Now, we let $y_n$ be a sequence of rational points decreasing to $x_*$ and $x_n$ be rational points increasing to $x_*$. By the above, the geodesics $\Gamma_{0}^{(x_n,i\varepsilon)}$ and $\Gamma_{\theta}^{(y_n,i\varepsilon)}$ are disjoint for all $n$. By sending $n\rightarrow \infty$ and taking subsequential limits of these geodesics and applying Proposition \ref{prop:cpt}, we obtain almost disjoint geodesics $\Gamma_0^p$ and $\Gamma_{\theta}^p$.

  Thus, due to the existence of the above $p$, we obtain that $t\in \dis_\theta\{J_{i,\varepsilon^{2/3}}\}$, and this completes the proof.

\end{proof}
With the above lemma at hand, the new goal is to obtain good upper bounds on $\PP(\dis_\theta\{J_{i,\varepsilon^{2/3}}\})$. We now follow the discussion at the end of Section \ref{sec:ideas} and give a $\varepsilon^{2/3-o(1)}$ bound on the above probability.
\begin{proposition}
  \label{lem:7}
 For any fixed $\alpha>1,\beta>0,\kappa\in (0,1/2)$, there exists a constant $C$ such that for all $\delta>0$, all $i\in [\![-\alpha \delta^{-3/2},-\alpha^{-1} \delta^{-3/2}]\!]$ and all $\theta\in (0,\log ^{1/2-\kappa}\delta^{-1})$,
  \begin{displaymath}
\PP\left(\dis_\theta\{J_{i,\delta}\}\right)\leq C \delta^{1-\beta}.
\end{displaymath}
\end{proposition}
\begin{proof}
Fix $t=i\delta^{3/2}$ and recall $\cX^t_\theta,\cD_\theta^t$ and $\mu_{\cD_\theta^t}$ from Proposition \ref{prop:7} and the definition of $\dis_\theta\{J_{i,\delta}\}$ from \eqref{eq:41}. On applying Proposition \ref{prop:7}, we obtain
  \begin{equation}
    \label{eq:47}
    \dis_\theta\{J_{i,\delta}\}= \{\cX^t_\theta\cap [-\delta\log^4\delta^{-1},\delta\log^4 \delta^{-1}]\neq \emptyset\}=\{\supp \mu_{\cD_\theta^t}\cap [-\delta\log^4\delta^{-1},\delta\log^4 \delta^{-1}]\neq \emptyset\}.
  \end{equation}
 Now recall the processes $\cP^{t,\theta}$ and $\cR^{t,\theta}$ from Proposition \ref{prop:1} and Lemma \ref{cor:ac} respectively. With the random constant $K=-\sup_{\xi\leq 0}(\cP_2^{t,\theta}(\xi)-\cP_1^{t,\theta}(\xi))$, an application of Proposition \ref{prop:1} yields
  \begin{equation}
    \label{eq:127}
    \cD_\theta^t(x)=\sup_{\xi\in (-\infty,x)}(\cR_2^{t,\theta}(\xi)-\cR_1^{t,\theta}(\xi))+K.
    \end{equation}
  With the notation $\tcR^{t,\theta}(x)=\cR^{t,\theta}(-x)$, we can write
  \begin{align}
    \label{eq:67}
    &\{\supp \mu_{\cD_\theta^t}\cap [-\delta\log^4\delta^{-1},\delta\log^4\delta^{-1}]\neq \emptyset\}\nonumber\\&\subseteq \{\exists x\in [-\delta\log^4\delta^{-1},\delta\log^4\delta^{-1}]:\cR^{t,\theta}_2(x)-\cR^{t,\theta}_1(x)=\sup_{\xi\in (-\infty,x)}(\cR^{t,\theta}_2(\xi)-\cR^{t,\theta}_1(\xi))\}\nonumber\\
    &\subseteq \{\exists x\in [-\delta\log^4\delta^{-1},\delta\log^4\delta^{-1}]:\cR^{t,\theta}_2(x)-\cR^{t,\theta}_1(x)=\sup_{\xi\in (-1/4,x)}(\cR^{t,\theta}_2(\xi)-\cR^{t,\theta}_1(\xi))\}\nonumber\\
    &\subseteq 
      \left\{
      \argmax( (\tcR^{t,\theta}_2-\tcR^{t,\theta}_1)\lvert_{[-\delta\log^4\delta^{-1},1/4]})\in [-\delta\log^4\delta^{-1},\delta\log^4\delta^{-1}]
      \right\}\eqqcolon\cA,
  \end{align}
  where the last line is used to define the event $\cA$. Note that the use of the notation ``$\argmax$'' in the last line is justified since by Lemma \ref{cor:ac}, $(\tcR_2^{t,\theta}-\tcR^{t,\theta}_1)\lvert_{[-1/4,1/4]}-(\tcR_2^{t,\theta}(0)-\tcR^{t,\theta}_1(0))$ is absolutely continuous to $(\sqrt{3}B +R)\lvert_{[-1/4,1/4]}$ and thus the above-mentioned maximum is almost surely attained at a unique point. Noting that $t=i\delta^{3/2}\in [-\alpha,-\alpha^{-1}]$ and applying Lemma \ref{cor:ac}, we obtain that for any $\nu>0$, there exists a constant $C$ such that for all $\delta>0$ and all $\theta\in (0,\log^{1/2-\kappa}\delta^{-1})$, we have 
  \begin{align}
    \label{eq:68}
    \PP(\cA)&\leq \left(\PP
    \left(
      \argmax( (\sqrt{3}B+R)\lvert_{[-\delta\log^4\delta^{-1},1/4]})\in [-\delta\log^4\delta^{-1},\delta\log^4\delta^{-1}] 
    \right)^{1-\nu}+\delta\right)\delta^{-2\nu}\nonumber\\
    &=\left(\PP
    \left(
      \argmax( (\sqrt{3}B+R)\lvert_{[-4\delta\log^4\delta^{-1},1]})\in [-4\delta\log^4\delta^{-1},4\delta\log^4\delta^{-1}] 
      \right)^{1-\nu}+\delta\right)\delta^{-2\nu}\nonumber\\
            &\leq \left(\left((4\delta\log^4\delta^{-1})\log(4\delta\log^4\delta)^{-1}\right)^{1-\nu}+\delta\right)\delta^{-2\nu}\nonumber\\
    &\leq C\delta^{1-4\nu},
  \end{align}
 where we used Corollary \ref{cor:max} to obtain the third line. At this point, we choose $\nu$ such that $4\nu<\beta$ and this completes the proof.
\end{proof}

All the ingredients are now in place to complete the proof of Lemma \ref{lem:5}.
\begin{proof}[Proof of Lemma \ref{lem:5}]
  By Lemma \ref{lem:2}, it suffices to show that for any fixed $\alpha>1$, $\theta>0$, $\beta>0$, we have for all $\varepsilon\in (0,1)$ and for all $i\in [\![-\alpha\varepsilon^{-1},-\alpha^{-1}\varepsilon^{-1}]\!]$,
  \begin{equation}
    \label{eq:121}
    \PP\left(\dis_\theta\{J_{i,\varepsilon^{2/3}}\}\right)+\PP\left(\cM_\theta[J_{i,\varepsilon^{2/3}}]\right)\leq C\varepsilon^{2/3-\beta}
  \end{equation}
for some constant $C$ depending on $\alpha,\theta,\beta$.  That the first term is bounded by $(C/2) \varepsilon^{2/3-\beta}$ follows by applying Proposition \ref{lem:7} with $\delta=\varepsilon^{2/3}$ and $\beta$ replaced by $3\beta/2$. On the other hand, the second term simply decays superpolynomially in $\varepsilon$ by Lemma \ref{lem:4} and thus can be bounded by $(C/2)\varepsilon^{2/3-\beta}$ as well. This yields \eqref{eq:121} and completes the proof.
\end{proof}
We now use Lemma \ref{lem:5} to upper bound the dimension of $\mult_\theta$, which is a subset of the set of $t\in (-\infty,0)$ such that $(\Gamma_0(t),t)$ is an atypical star, and this in conjunction with Lemma \ref{lem:200} will complete the proof of Theorem \ref{thm:1}.
\begin{lemma}
  \label{prop:6}
  For each fixed $\theta\neq 0$, we almost surely have $\dim\mult_{\theta}\leq 1/3$.
\end{lemma}

\begin{proof}
  By the symmetries of the directed landscape (see \cite[Proposition 1.23]{DV21}), it suffices to consider $\theta>0$.
  By using a countable union argument over an increasing sequence of $\alpha$s, we can further reduce to showing that for any $\alpha>1$ and any $\theta>0$, $\dim (\mult_{\theta}\cap [-\alpha,-\alpha^{-1}])\leq 1/3$ almost surely.
 Note that for any $\varepsilon\in (0,1)$,
  \begin{equation}
    \label{eq:9}
    \mult_{\theta}\cap[\alpha,1-\alpha]\subseteq \bigcup_{i=-\alpha \varepsilon^{-1}}^{-\alpha^{-1}\varepsilon^{-1}}\mult_{\theta}\cap[i\varepsilon,(i+1)\varepsilon].
  \end{equation}
  If we let $N_\varepsilon$ denote the number of such $i\in [\![ -\alpha \varepsilon^{-1},-\alpha^{-1}\varepsilon^{-1}]\!]$ such that $\mult_{\theta}\cap[i\varepsilon,(i+1)\varepsilon]\neq \emptyset$, then by applying Lemma \ref{lem:5}, we obtain that $\EE N_\varepsilon\leq C \varepsilon^{2/3-\beta} \times \varepsilon^{-1}=C\varepsilon^{-1/3-\beta}$ for any fixed $\beta>0$ and all small enough $\varepsilon$. By the usual argument for Hausdorff dimension upper bounds (see for e.g. \cite[Lemma 6.10.5]{BP16}), this
 yields that a.s.\ $\dim \mult_{\theta}\cap [-\alpha,-\alpha^{-1}]\leq 1/3+\beta$. Now we note that Lemma \ref{lem:5} holds for all $\beta>0$, thereby showing that a.s.\ $\dim \mult_{\theta}\cap [-\alpha,-\alpha^{-1}]\leq 1/3$, and this completes the proof.
\end{proof}
\begin{proof}[Proof of Theorem \ref{thm:1}]
  The lower bound follows by Lemma \ref{lem:200} and the upper bound by Lemma \ref{prop:6}.
\end{proof}
A countable union argument can now be used along with Lemma \ref{prop:6} to obtain the dimension of all non reverse $1$-stars on $\Gamma_0$.
\begin{lemma}
  \label{lem:1}
  Almost surely, for every interval $I=(-\delta,0)$ with $\delta>0$, the set of $t\in I$ such that $(\Gamma_0(t),t)$ is not a reverse $1$-star has Hausdorff dimension $1/3$.
\end{lemma}
\begin{proof}
We locally define $\cS$ by
  \begin{equation}
    \label{eq:124}
    \cS=\{t\in (-\infty,0):(\Gamma_0(t),t) \text{ is not a reverse $1$-star}\}.
  \end{equation}
Since $\mult_\theta\subseteq \cS$ for all $\theta\neq 0$, Lemma \ref{lem:200} immediately implies that for every $\delta>0$, $\dim \cS\cap (-\delta,0)\geq 1/3$ almost surely. For the upper bound, in view of Lemma \ref{prop:6}, it suffices to show that $\cS\subseteq \bigcup_{\theta\in \QQ}\mult_{\theta}$. To see this, suppose $t\in \cS$. That is, there exist $(x_1,s),(x_2,s)$ with $x_1<x_2$ and $s<t$ such that there are two almost disjoint geodesics $\Gamma_{(x_1,s)}^{(\Gamma_0(t),t)},\Gamma_{(x_2,s)}^{(\Gamma_0(t),t)}$. Note that we can assume that one of $x_1,x_2$ equals $\Gamma_0(s)$. Here, we take $x_1=\Gamma_0(s)$ and show that $t\in \bigcup_{\theta\in \QQ_+}\mult_{\theta}$ in this case. In the other case $x_2=\Gamma_0(s)$, the proof would be the same but we would instead have $t\in \bigcup_{-\theta\in \QQ_+}\mult_{\theta}$.

 Now as a consequence of Proposition \ref{lem:45}, there exists a random angle $\theta_0$ such that $\Gamma_{(x,-r)}^{(\Gamma_0(t),t)}(s)>x_2$ for all $x>r\theta_0$ and all large $r$. By the definition of a $\theta$-directed downward semi-infinite geodesic, for any $\theta>\theta_0$, $\Gamma^{(\Gamma_0(t),t)}_\theta(-r)>r\theta_0$ for all large enough $r$ depending on $\theta$. As a consequence, by choosing an $r$ large enough, %
 $\Gamma^{(\Gamma_0(t),t)}_\theta(s)>\Gamma^{(\Gamma_0(t),t)}_{(r\theta_0,-r)}(s)>x_2$ for all $\theta>\theta_0$. %
 With this choice,
 we obtain that for any $\theta>\theta_0$, all geodesics $\Gamma_\theta^{(\Gamma_0(t),t)}$ satisfy $\Gamma_\theta^{(\Gamma_0(t),t)}(r)\geq \Gamma_{(x_2,s)}^{(\Gamma_0(t),t)}(r)$ for all $r\in [s,t)$, and we now fix such a geodesic $\Gamma_\theta^{(\Gamma_0(t),t)}$.

 Thus by using the almost disjointness of $\Gamma_{(\Gamma_0(s),s)}^{(\Gamma_0(t),t)},\Gamma_{(x_2,s)}^{(\Gamma_0(t),t)}$, we obtain that $\Gamma_0^{(\Gamma_0(t),t)}\lvert_{(s,t]}=\Gamma_0\lvert_{(s,t]}$ and $\Gamma_\theta^{(\Gamma_0(t),t)}\lvert_{(s,t]}$ are almost disjoint as well. This in fact implies that $\Gamma_0^{(\Gamma_0(t),t)}, \Gamma_\theta^{(\Gamma_0(t),t)}$ are almost disjoint since otherwise there would exist another $\theta$-directed downward semi-infinite geodesic $\gamma$ emanating from $(\Gamma_0(t),t)$ satisfying $\gamma(s)=\Gamma_0(s)<x_2$, and this would contradict the last line of the previous paragraph. Thus $t\in \mult_{\theta}$ for this choice of $\theta$, and this shows $\cS\subseteq \bigcup_{\theta\in \QQ} \mult_{\theta}$. By using the above along with Lemma \ref{prop:6}, we obtain that $\dim \cS\leq 1/3$ almost surely, and this completes the proof.
\end{proof}
We now combine Lemma \ref{lem:1} with an absolute continuity argument to finish the proof of Theorem \ref{thm:2}.
\begin{proof}[Proof of Theorem \ref{thm:2}]
  
  The idea is to use an absolute continuity argument to compare finite segments of semi-infinite geodesics with finite geodesics; a similar argument to the one we present, was used in \cite[Lemma 3.5]{Dau22} to transfer between different initial conditions. By using the symmetries of the directed landscape \cite{DOV18}, it suffices to show that the set of $t\in (-1,0)$ such that $(\Gamma_{(0,-1)}^{(0,0)}(t),t)$ is not a reverse $1$-star a.s.\ has dimension $1/3$. By the stability of Hausdorff dimensions under countable unions, it suffices to fix an $\alpha\in (0,1/2)$ and show that the set $\wcC$ defined by
  \begin{equation}
    \label{eq:7}
    \wcC=\left\{t\in(-1+\alpha,0)\colon \Gamma_{(0,-1)}^{(0,0)}\textrm{ is not a reverse } 1\textrm{-star}\right\}
  \end{equation}
  satisfies $\dim \wcC= 1/3$ almost surely. %
  We now define the random initial data $f,\widetilde{f}\colon \RR\rightarrow \RR$ by $\widetilde{f}(x)=\cL(0,-1;x,-1+\alpha), f(x)=\cB_0^{-1+\alpha}(x)$ and note that
  $\Gamma_0\lvert_{[-1+\alpha,0]}=\Gamma_{(f,-1+\alpha)}^{(0,0)}$, where the latter is the almost surely unique geodesic corresponding to the passage time
  \begin{equation}
\label{eq:4}
    \cL(f,-1+\alpha;0,0)\coloneqq\sup_{x \in \RR}\{f(x)+\cL(x,-1+\alpha;0,0)\}.
  \end{equation}
  Similarly, $\Gamma_{(0,-1)}^{(0,0)}\lvert_{[-1+\alpha,0]}$ is the a.s.\ unique geodesic corresponding to the analogously defined passage time $\cL(\widetilde{f},-1+\alpha;0,0)$. If we define $\cC$ to be the set of $t\in(-1+\alpha,0)$ such that $\Gamma_0(t)$ is not a reverse $1$-star, then by Lemma \ref{lem:1}, $\dim \cC= 1/3$ almost surely. With $\cL\lvert_{[-1+\alpha,0]}$ denoting the restriction of the landscape to points with time coordinates in $[-1+\alpha,0]$, we obtain that for any random initial condition $h\colon \RR\rightarrow \RR$ which is absolutely continuous to $f$ and independent of $\cL\lvert_{[-1+\alpha,0]}$, the set $\cC_h$ of $t\in (-1+\alpha,0)$ with $(\Gamma_{(h,-1+\alpha)}^{(0,0)}(t),t)$ not being a reverse $1$-star satisfies $\dim \cC_h= 1/3$ almost surely.

     With $B$ being a Brownian motion independent of $\cL$, define the initial conditions $h_n$ such that $h_n\lvert_{[-n,n]}=\widetilde{f}\lvert_{[-n,n]}$, $h_n(n+x)=\widetilde{f}(n)+\sqrt{2}B(x)$ for all $x>0$ and $h_n(-n-x)=\widetilde{f}(-n)+\sqrt{2}B(-x)$ for all $x>0$ as well. Using that $\widetilde{f}$ is the parabolic $\mathrm{Airy}_2$ process which is locally Brownian \cite{CH14}, It is straightforward to observe that for each fixed $n$, $h_n$ is absolutely continuous to $f$, a Brownian motion of diffusivity $2$ and is moreover independent of $\cL\lvert_{[-1+\alpha,0]}$.  Thus by the discussion from the previous paragraph, $\dim\cC_{h_n}= 1/3$ for all $n\in \NN$ almost surely. We now claim that a.s.\ for all $n$ large enough,
    \begin{equation}
      \label{eq:10}
      \Gamma_{(h_n,-1+\alpha)}^{(0,0)}=\Gamma_{(\widetilde{f},-1+\alpha)}^{(0,0)}=\Gamma_{(0,-1)}^{(0,0)}\lvert_{[-1+\alpha,0]}.
    \end{equation}
    Proving the above claim would complete the proof since it would imply that a.s.\ $\wcC=\cC_{h_n}$ for all large $n$, and we already know that $\dim \cC_{h_n}=1/3$ simultaneously for all $n\in \NN$.

    We now show the first equality in \eqref{eq:10}; note that the latter equality follows just by the definition of the initial condition $\widetilde{f}$. Since $h_n\lvert_{[-n,n]}=\widetilde{f}\lvert_{[-n,n]}$, it suffices to show that $\Gamma_{(h_n,-1+\alpha)}^{(0,0)}(-1+\alpha)\in [-n,n]$
    for all large $n$ almost surely. %
    By \cite[Corollary 10.7]{DOV18}, there exists a random constant $C$ measurable with respect to $\sigma(\cL)$ such that for all $x\in \RR$,
    \begin{equation}
      \label{eq:11}
      |\cL(x,-1+\alpha;0,0)+x^2/(1-2\alpha)|\leq C\log^2(4(|x|+2)).
    \end{equation}
    Thus for any $n$ with $n\geq n_0$, and for any $x\in[-n,n]^c$,
    \begin{align}
      \label{eq:12}
      h_n(x)+\cL(x,-1+\alpha;0,0)&\leq h_n(x)-x^2/(1-2\alpha)+C\log^2(4(|x|+2))\nonumber\\
                                            &\leq h_n(x)-cx^2\nonumber\\
      &\leq \max(\widetilde{f}(-n),\widetilde{f}(n))+\sup_{x\in \RR}(B(x)-c(n+|x|)^2)
    \end{align}
    for some deterministic constant $c$. By using that $\widetilde{f}$, being a parabolic $\mathrm{Airy}_2$ process, stays bounded, it is not difficult to see that the final term in \eqref{eq:12} goes to $-\infty$ as $n\rightarrow \infty$ and this shows that
    \begin{equation}
      \label{eq:14}
      \sup_{x\in [-n,n]^c}(h_n(x)+\cL(x,-1+\alpha;0,0))\rightarrow -\infty
    \end{equation}
    as $n\rightarrow \infty$. Thus for all $n$ large enough,
    \begin{equation}
      \label{eq:44}
      \sup_{x\in [-n,n]^c}(h_n(x)+\cL(x,-1+\alpha;0,0))< \widetilde{f}(0)+\cL(0,-1+\alpha;0,0),
    \end{equation}
    and as a consequence, $\Gamma_{(h_n,-1+\alpha)}^{(0,0)}(-1+\alpha)\in [-n,n]$ for all large $n$. This completes the proof.

\end{proof}

\section{An open question}
\label{sec:open}
 In this section, we discuss a follow up question which arises from this paper. To motivate the discussion, we recall the uncentered spatial difference profile $\fD^{\mathrm{un}}$ from \eqref{eq:128} and state a result from \cite{GH21} which gives global and local comparisons of $\fD^{\mathrm{un}}\lvert_{[c,d]}$, for any interval $[c,d]$, to Brownian local time of rate $4$.
  \begin{proposition}
    \label{thm:7}
    Let $L$ be the local time process corresponding to a Brownian motion of rate $4$ and let $M$ be its running maximum process, both viewed as processes on $\RR_+$. Recall that by Levy's identity, $L\stackrel{d}{=}M$. For any interval $[c,d]\subseteq \RR$, $\fD^{\mathrm{un}}\lvert_{[c,d]}-\fD^{\mathrm{un}}(c)$ is absolutely continuous to $L\lvert_{[c,d]}-L(c)$.
    Further, if we sample a point $z$ from the measure induced by $\fD^{\mathrm{un}}\lvert_{[c,d]}$, then the profile on $\RR_+$ defined by $x\mapsto \varepsilon^{-1/2}\fD^{\mathrm{un}}[(z+\varepsilon x)-\fD^{\mathrm{un}}(z)]$ converges in distribution as $\varepsilon\rightarrow 0$ to $L$.
  \end{proposition}
  One might wonder if similar investigations are possible for the geodesic difference profile $\cD_\theta$ instead of the spatial profile $\fD^{\mathrm{un}}$. Though interesting, investigating the local and global behaviour of $\cD_\theta$ seems difficult since much less is known about even the temporal process $\cL(0,0;0,t)$ compared to the spatial process $\cL(0,0;x,1)$. Indeed, the latter is just the $\text{Airy}_2$ process and is amenable to a host of integrable techniques. One the other hand, one might attempt to investigate the local behaviour of a spatial distance profile around a typical point of increase of the geodesic difference profile, and in this case, the conjectural local limit is very concrete.

  \begin{conjecture}
    Fix $\theta>0$ and let $\mu_\theta$ denote the Lebesgue-Stieltjes measure induced on $(-\infty,0)$ by the monotonically increasing geodesic difference profile $\cD_\theta$. Show that for any fixed interval $I\subseteq (-\infty,0)$ and a time $\tau$ sampled according to $\mu\lvert_{I}$, the profile on $\RR_+$ defined by $x\mapsto \varepsilon^{-1/2}[\cD^\tau_\theta(\varepsilon x)-\cD_\theta^\tau(0)]$ converges in law as $\varepsilon\rightarrow 0$ to the process $S$ defined by $S(x)=\max_{0\leq y\leq x}(\sqrt{3}B(y)+R(y))$.
\end{conjecture}
In other words, while the spatial Busemann difference profile around its typical points of increase looks like the running maximum of a Brownian motion of diffusivity $4$ by Proposition \ref{thm:7}, it should instead behave like the running maximum of $\sqrt{3}B+R$ around a typical point of increase of $\cD_\theta$. Since the above problem partly involves spatial difference profiles, we hope that it might be more amenable to analysis. However, the proof of the local limit result in Proposition \ref{thm:7} relied on the \textit{strong Brownian Gibbs property} \cite[Lemma 2.5]{CH14} of the Airy line ensemble and these tools do not seem to directly apply in the setting of the above question. %

\section{Appendix: The proof of Lemma \ref{cor:ac}}
In this section, we provide the proof of Lemma \ref{cor:ac} which we had postponed earlier. In fact, we will instead prove the following stronger lemma.
\begin{lemma}
  \label{lem:31}
  For any fixed $\theta>0$, $\cR^{t,\theta}\lvert_{[-1/4,1/4]}-\cR^{t,\theta}(0)$ is absolutely continuous to $(B_1-R,\sqrt{2} B_2)\lvert_{[-1/4,1/4]}$. Further, for any fixed $\nu>0,\alpha>1,\kappa\in (0,1/2)$, there exists a constant $C$ such that the following holds for all $\delta>0,\phi\in (0,1)$ and all $\theta\in (0,\log^{1/2-\kappa}\delta^{-1})$. For any measurable set $E$ satisfying $\PP((B_1-R,\sqrt{2} B_2)\lvert_{[-1/4,1/4]}\in E)=\phi$, we have
  \begin{equation}
    \label{eq:5}
    \PP(\cR^{t,\theta}\lvert_{[-1/4,1/4]}-\cR^{t,\theta}(0)\in E)\leq C(\phi^{1-\nu}\delta^{-2\nu}+\delta^{1-2\nu})
  \end{equation}
  for all $t\in [-\alpha,-\alpha^{-1}]$.
\end{lemma}
\begin{proof}[Proof of Lemma \ref{cor:ac} assuming Lemma \ref{lem:31}]
  Apply Lemma \ref{lem:31} with the measurable set $E$ defined such that $(B_1-R,\sqrt{2} B_2)\lvert_{[-1/4,1/4]}\in E$ if and only if $(\sqrt{2}B_2-(B_1-R))\lvert_{[-1/4,1/4]} \in H$.
\end{proof}
We spend the rest of the appendix proving Lemma \ref{lem:31}. For $t<0$, we define $\mathrm{Ai}^t(x)=\cL(x,t;0,0)$ to keep the notation clean. We first state a few lemmas and then combine them to complete the proof of Lemma \ref{lem:31}.

\begin{lemma}
  \label{cor:43}
  For any fixed $\nu>0,\alpha>1,\kappa\in (0,1/2)$, there exists a constant $C$ such that for all $\delta>0,\phi\in (0,1)$, the following holds for all $d\in [ -\log^{1/3}\delta^{-1},\log^{1/3}\delta^{-1}]$ and $\theta\in (0,\log^{1/2-\kappa}\delta^{-1})$. For any measurable set $A$ satisfying $\PP((\sqrt{2}B_0,\sqrt{2}B_1,\sqrt{2}B_2)\lvert_{[0,1]}\in A)=\phi$, we have
  \begin{equation}
    \label{eq:3}
    \PP( (\mathrm{Ai}^t(d+\cdot)-\mathrm{Ai}^t(d),\cP^{t,\theta}_1(d+\cdot)-\cP^{t,\theta}_1(d),\cP_2^{t,\theta}(d+\cdot)-\cP_2^{t,\theta}(d))\lvert_{[0,1]}\in A)\leq C(\phi^{1-\nu} \delta^{-\nu}+\delta^{1-\nu})
  \end{equation}
  for all $t\in [-\alpha,-\alpha^{-1}]$. 
 \end{lemma}
 \begin{proof}
   By using the KPZ scaling of the directed landscape \cite[(5) in Lemma 10.2]{DOV18}, it is straightforward to obtain that $\mathrm{Ai}^t(x)\stackrel{d}{=}|t|^{1/3}\mathrm{Ai}^1(|t|^{-2/3}x)$ as a process in $x$. By using the above along with Proposition \ref{prop:airy}, it can be shown that there exists a constant $C$ such that for any $\delta>0$, any measurable set $E$ with $\PP(\sqrt{2}B\lvert_{[0,1]}\in E)\geq\delta$, for all $t\in [-\alpha,\alpha^{-1}]$,
   \begin{equation}
     \label{eq:25}
     \PP((\mathrm{Ai}^t(d+\cdot)-\mathrm{Ai}^t(d))\lvert_{[0,1]}\in E)\leq C\PP(\sqrt{2}B\lvert_{[0,1]}\in E)\delta^{-\nu/2}
   \end{equation}
  for all $d\in [-\log^{1/3}\delta^{-1},\log^{1/3}\delta^{-1}]$. Thus if we fix $F$ to be a measurable set with $\PP(\sqrt{2}B\lvert_{[0,1]}\in F)=\delta$, then for any measurable set $S$, %
   \begin{align}
     \label{eq:15}
     \PP((\mathrm{Ai}^t(d+\cdot)-\mathrm{Ai}^t(d))\lvert_{[0,1]}\in S)&\leq \PP((\mathrm{Ai}^t(d+\cdot)-\mathrm{Ai}^t(d))\lvert_{[0,1]}\in S\cup F)\nonumber\\
                                                                        &\leq C\delta^{-\nu/2}(\PP(\sqrt{2}B\lvert_{[0,1]}\in S)+\PP(\sqrt{2}B\lvert_{[0,1]}\in F))\nonumber\\
     &\leq C\delta^{-\nu/2}(\PP(\sqrt{2}B\lvert_{[0,1]}\in S)+\delta).
   \end{align}
   By the independence built into the directed landscape along with Proposition \ref{prop:1}, $\mathrm{Ai}^t,\cP_1^{t,\theta},\cP_2^{t,\theta}$ are mutually independent and further, $\cP_1^{t,\theta}(x)$ is exactly a Brownian motion of diffusivity $2$ while $\cP_2^{t,\theta}(x)\stackrel{d}{=}\sqrt{2}B_2(x)+2\theta x$.

   Now, a Brownian motion with drift is absolutely continuous to Brownian motion on unit intervals. In fact, if we use $\mu$ to denote the law of $\sqrt{2} B\lvert_{[0,1]}$ and $\nu$ to denote the law of $x\mapsto \sqrt{2} B(x)+2\theta x$ on the same interval, then the Radon Nikodym derivative $\frac{d\nu}{d\mu}$ is given by
   \begin{equation}
     \label{eq:13}
     \frac{d\nu}{d\mu}(\sqrt{2}B\lvert_{[0,1]})= e^{(\sqrt{2}B(1))^2/4}e^{-(\sqrt{2}B(1)-2\theta)^2/4}=e^{-\theta^2+\sqrt{2}\theta B(1)},
   \end{equation}
   and thus for $p>1$, we can estimate the $p$th moment of the above by
   \begin{align}
     \label{eq:16}
     \EE\left[(\frac{d\nu}{d\mu}(\sqrt{2}B\lvert_{[0,1]}))^p\right]=e^{-\theta^2 p}\EE_{X\sim N(0,1)} e^{\sqrt{2}p\theta X}&=e^{\theta^2(p^2-p)}\nonumber\\
     &\leq \exp((p^2-p)\log^{1-2\kappa}\delta^{-1}),
   \end{align}
   where the last inequality above holds for all $\theta\in (0,\log^{1/2-\kappa}\delta^{-1})$. The important aspect here is that the right hand side above grows subpolynomially in $\delta^{-1}$ for any fixed $p$.

  As a result of the above along with H\"older's inequality, we obtain that if we couple the Brownian motions $B_0,B_1,B_2$ to be independent of $\cL$, then for any measurable set $A$, any $p>1$ and all $\theta\in (0,\log^{1/2-\kappa}\delta^{-1})$, we have
   \begin{align}
     \label{eq:19}
     &\PP( (\sqrt{2}B_0,\cP^{t,\theta}_1(d+\cdot)-\cP^{t,\theta}_1(d),\cP_2^{t,\theta}(d+\cdot)-\cP_2^{t,\theta}(d))\lvert_{[0,1]}\in A)\nonumber\\
     &\leq \EE\left[(\frac{d\nu}{d\mu}(\sqrt{2}B\lvert_{[0,1]}))^p\right]^{1/p}\PP( (\sqrt{2}B_0,\sqrt{2}B_1,\sqrt{2}B_2)\lvert_{[0,1]}\in A)^{(p-1)/p}
   \end{align}
   Now, by fixing a choosing a large enough $p$ and using \eqref{eq:16}, we obtain that fixed $\nu>0$, for all small enough $\delta$ and all $\theta\in (0,\log^{1/2-\kappa}\delta^{-1})$, 
   \begin{equation}
     \label{eq:26}
     \PP( (\sqrt{2}B_0,\cP^{t,\theta}_1(d+\cdot)-\cP^{t,\theta}_1(d),\cP_2^{t,\theta}(d+\cdot)-\cP_2^{t,\theta}(d))\lvert_{[0,1]}\in A)\leq \delta^{-\nu/2}\PP( (\sqrt{2}B_0,\sqrt{2}B_1,\sqrt{2}B_2)\lvert_{[0,1]}\in A)^{1-\nu}.
   \end{equation}
   
   Combining \eqref{eq:15} with \eqref{eq:26}, we obtain
   \begin{align}
     \label{eq:18}
     &\PP( (\mathrm{Ai}^t(d+\cdot)-\mathrm{Ai}^t(d),\cP^{t,\theta}_1(d+\cdot)-\cP^{t,\theta}_1(d),\cP_2^{t,\theta}(d+\cdot)-\cP_2^{t,\theta}(d))\lvert_{[0,1]}\in A)\nonumber \\
     &\leq C\delta^{-\nu/2}(\PP( (\sqrt{2}B_0,\cP^{t,\theta}_1(d+\cdot)-\cP^{t,\theta}_1(d),\cP_2^{t,\theta}(d+\cdot)-\cP_2^{t,\theta}(d))\lvert_{[0,1]}\in A)+\delta)\nonumber\\
     &\leq C\delta^{-\nu/2}(\delta^{-\nu/2}\PP( (\sqrt{2}B_0,\sqrt{2}B_1,\sqrt{2}B_2)\lvert_{[0,1]}\in A)^{1-\nu}+\delta)\nonumber\\
     &\leq C_1(\delta^{1-\nu}+\phi^{1-\nu}\delta^{-\nu})
   \end{align}
   for some constant $C_1$, and this completes the proof.
 \end{proof}
 A fact that we shall use shortly is that a Brownian motion on an interval around its minimum looks like two independent scaled Brownian meanders \cite{Den84}. The following lemma gives an absolute continuity estimate for a Brownian meander with respect to the Bessel-$3$ process.

\begin{lemma}
  \label{lem:34}
  Let $H$ be a Brownian meander. Then $H$ is absolutely continuous to the law of $R\lvert_{[0,1]}$. Further, there exists a constant $C$ such that for any measurable set $F$, we have $\PP(H\lvert_{[0,1/2]}\in F)\leq C \PP(R\lvert_{[0,1/2]}\in F)$.
\end{lemma}

\begin{proof}
 It is a fact \cite{Imh84,RY15} that $H$ can be sampled by first sampling $H(1)$, which has the Rayleigh density $xe^{-x^2/2}$ and then drawing a Bessel-$3$ bridge from $(0,0)$ to $(H(1),1)$. Similarly, $R$ can be described by first sampling $R(1)$ which has the density $\sqrt{2/\pi}x^2e^{-x^2/2}$ and then drawing a Bessel-$3$ bridge between the obtained points. Thus the Radon-Nikodym derivative of $H$ with respect to $R$ is $\sqrt{\pi/2}R(1)^{-1}$ which is bounded for $R(1)>1$. Denoting a Bessel-$3$ bridge from $0$ to $y$ on the time interval $[0,1]$ by $\mathfrak{B}_y$, it thus suffices to show that there exists a constant $C$ such that for any $y\in [0,1]$ and any measurable set $F$, $\PP(\mathfrak{B}_y\lvert_{[0,1/2]}\in F)\leq C\PP(R\lvert_{[0,1/2]}\in F)$.

  Now note that $\mathfrak{B}_y$ can be described as the absolute value of $Y_1$, a $3$-dimensional Brownian Bridge from $0$ to $(y,0,0)$ with $R$ similarly being the absolute value of $Y_2$, a $3$-dimensional Brownian motion started from $0$. It thus suffices to show that for any measurable set $F$, $\PP(Y_1\lvert_{[0,1/2]}\in F)\leq C\PP(Y_2\lvert_{[0,1/2]}\in F)$ for all $y\in [0,1]$. To obtain this, we note that $Y_1$ can be described as first sampling the point $Y_1(1/2)$, which is distributed as a $3$-dimensional centered Gaussian with the covariance matrix $\Sigma_1=(1/4)\mathbf{Id}_3$ and mean $(y/2,0,0)$, and then further drawing a Brownian Bridge from $(0,0,0)$ to $Y_1(1/2)$ on the time interval $[0,1/2]$. Similarly, to sample $Y_2\lvert_{[0,1/2]}$, we first sample $Y_2(1/2)$ which is distributed as a $3$-dimensional Gaussian with the covariance matrix $\Sigma_2=\mathbf{Id}_3$ and then similarly sampling a Brownian Bridge. On by writing down the densities explicitly, it can be seen that the Radon Nikodym derivative of the Gaussian corresponding to $\Sigma_1$ with respect to the Gaussian corresponding to $\Sigma_2$ is bounded. Thus there exists a constant $C$ such that $\PP(\mathfrak{B}_y\lvert_{[0,1/2]}\in F)\leq C\PP(R\lvert_{[0,1/2]}\in F)$ for all $y\in [0,1]$ and this completes the proof.
\end{proof}
The following simple consequence of the above lemma will be useful to us.
\begin{lemma}
  \label{lem:44}
  For $\ell,r>0$, suppose that $g_1\colon[-2\ell,0]\rightarrow \RR_+$ and $g_2\colon[0,2r]\rightarrow \RR_+$ are random functions satisfying that $\frac{1}{\sqrt{2\ell}}g_1(-2\ell\cdot)$ and $\frac{1}{\sqrt{2r}}g_2(2r\cdot)$ are independent Brownian meanders. Let $g=g_1\bullet g_2$ be defined by $g(x)=g_2(x)$ on $[0,2r]$ and $g(x)=g_1(x)$ on $[-2\ell,0]$. Then there exists a constant $C$ such that for any measurable set $F$ and any $\ell,r$ as above,
  \begin{equation}
    \label{eq:114}
    \PP(g\lvert_{[-\ell,r]}\in F)\leq C\PP(R\lvert_{[-\ell,r]}\in F).
  \end{equation}
\end{lemma}
\begin{proof}
  Couple $R,g_1,g_2$ such that the three are mutually independent. Also, we define $R_1\colon[-\infty,0]\rightarrow \RR_+$ and $R_2\colon[0,\infty]\rightarrow \RR_+$ such that $R=R_1\bullet R_2$; note that $R_1$ and $R_2$ are independent. We now have
  \begin{align}
    \label{eq:115}
    \PP(g\lvert_{[-\ell,r]}\in F)&= \EE\left[\PP(
      g_1\lvert_{[-\ell,0]}\bullet g_2\lvert_{[0,r]}\in F\lvert g_2)
                 \right]\leq C\EE\left[\PP(R_1\lvert_{[-\ell,0]}\bullet g_2\lvert_{[0,r]}\in F\lvert g_2)\right]\nonumber\\&=C\EE\left[\PP(R_1\lvert_{[-\ell,0]}\bullet g_2\lvert_{[0,r]}\in F\lvert R_1)\right]
    \leq C^2\EE\left[\PP(R_1\lvert_{[-\ell,0]}\bullet R_2\lvert_{[0,r]}\in F\lvert R_1)\right]\nonumber\\&=C^2\PP(R\lvert_{[-\ell,r]}\in F).
  \end{align}
  In the above, we have used Lemma \ref{lem:34} along with the fact that $R$ respects Brownian scaling, that is, for any $\alpha>0$, $\frac{1}{\sqrt{\alpha}}R(\alpha\cdot)\stackrel{d}{=}R$.
\end{proof}
We are finally ready to complete the proof of Lemma \ref{lem:31}.
\begin{proof}[Proof of Lemma \ref{lem:31}]
 Recall the definition of $\cR^{t,\theta}(x)$ from the statement of Lemma \ref{cor:ac}. Note that the process $\cR^{t,\theta}$ is centered around $\Gamma_0(t)$ which can be calculated as $\Gamma_0(t)=\argmax_x \{\mathrm{Ai}^t(x) +\cP^{t,\theta}_1(x)\}$, where we use $\mathrm{Ai}^t(x)$ to denote $\cL(x,t;0,0)$. Let $B'$ be a standard Brownian motion. In this proof, we work with the coupling where $B,B',R,\{B_i\}_{i\geq 0},\cP^{t,\theta},\mathrm{Ai}^t$ are mutually independent.
 We note that by the transversal fluctuation estimates in Proposition \ref{prop:8}, %
  $\PP(|\Gamma_0(t)|>x)\leq Ce^{-cx^3}$ for all $x>0$, and thus for some constant $C'$ and all $\delta>0$,
  \begin{equation}
    \label{eq:74}
    \PP(|\Gamma_0(t)|>C'\log^{1/3}\delta^{-1})\leq \delta.
  \end{equation}
  We now define the overlapping intervals $I_i=[\ell_i,r_i]$, where $\ell_i=-C'\log^{1/3}\delta^{-1}+(i/2)$ and $r_i=\ell_i+1$
  for $i\in [\![0,2C'\log^{1/3}\delta^{-1}]\!]$. We note that on the event $\{|\Gamma_0(t)|\leq C'\log^{1/3}\delta^{-1}\}$, there exists an $i$ such that $\Gamma_0(t)\in I_i$ and $(\Gamma_0(t)-\ell_i),(r_i-\Gamma_0(t))\in [1/4,3/4]$.
  Also, throughout this proof, we use the notation
  \begin{equation}
    \label{eq:43}
    \fB=(\mathrm{Ai}^t,\cP^{t,\theta}_1,\cP^{t,\theta}_2).
  \end{equation}

  Let $M_i=\max_{x\in I_i}\{\cP_1^{t,\theta}(x)+\mathrm{Ai}^t(x)\}$ and let $m_i$ denote the almost surely unique location where the maximum is attained. We now show that the following holds for all $\psi\in (0,1),\delta>0$ and $\theta\in (0,\log^{1/2-\kappa}\delta^{-1})$: there exists a constant $C$ such that for any $m_i$ satisfying $(m_i-\ell_i),(r_i-m_i)\in [1/4,3/4]$, and for any measurable set $F$ with
  \begin{equation}
    \label{eq:110}
    \PP\left((-B-R,B-R,\sqrt{2}B')\lvert_{[-1/4,1/4]}\in F\right)=\psi,
  \end{equation}
 we have
  \begin{equation}
    \label{eq:70}
    \PP
    \left(
      (\fB(m_i+\cdot)-\fB(m_i))\lvert_{[-1/4,1/4]}\in F\lvert m_i
    \right)\leq C(\psi^{1-\nu}\delta^{-\nu}+\delta^{1-\nu}).
  \end{equation}
  
  Let $\underline{M}= \max_{x\in [0,1]}(B_0(x)+B_1(x))$ and let $\underline{m}$ denote the corresponding $\argmax$. We define $\underline{\fB}=(\sqrt{2}B_0,\sqrt{2}B_1,\sqrt{2}B_2)$.
   We first show that there exists a constant $C$ such that for any $\underline{m}\in [1/4,3/4]$,
    \begin{equation}
    \label{eq:71}
    \PP
    \left(\left(
      \underline{\fB}(\underline{m}+\cdot)-\underline{\fB}(\underline{m})\right)\lvert_{[-1/4,1/4] }\in F\lvert \underline{m}
    \right)\leq C\psi,
  \end{equation}
  and then just use the above to obtain \eqref{eq:70} as we shall explain later.
  Now note the fact \cite{Den84} that for a Brownian motion $W$, if $M=\max_{x\in [0,1]}W(x)$ and if $m$ denotes the almost surely unique point where the above maximum is attained, then conditional on $m$, the random functions $h_1^m,h_2^m\colon[0,1]\rightarrow \RR_+$ defined by $h^m_1(x)= \frac{1}{\sqrt{m}}(W(m)-W(m-mx))$ and $h_2^m(x)=\frac{1}{\sqrt{1-m}}(W(m)-W(m+(1-m)x))$ are independent Brownian meanders.
 Thus if we define $g^{\underline{m}}\colon[-\underline{m},1-\underline{m}]$ by $g^{\underline{m}}(x)=(B_0+B_1)(\underline{m})-(B_0+B_1)(\underline{m}+x)$ and also define $g_1^{\underline{m}}=g^{\underline{m}}\lvert_{[-\underline{m},0]},g_2^{\underline{m}}=g^{\underline{m}}\lvert_{[0,1-\underline{m}]}$, then conditional on $\underline{m}$, $\frac{1}{\sqrt{m}}g_1^{\underline{m}}(-\underline{m}\cdot)$ and $\frac{1}{\sqrt{1-\underline{m}}}g_2^{\underline{m}}((1-\underline{m})\cdot)$ are $\sqrt{2}$ times independent Brownian meanders. Now, we can use Lemma \ref{lem:44} to obtain that there exists a constant $C$ such that for any measurable set $H$, on the event $\{\underline{m}\in [1/4,3/4]\}$, %
  \begin{equation}
    \label{eq:72}
    \PP
    \left(
      \left( (B_0+B_1)(\underline{m})- (B_0+B_1)(\underline{m}+\cdot)  \right)\lvert_{[-1/4,1/4]}\in H
    \lvert \underline{m}\right)\leq C\PP
    \left(
      \sqrt{2}R\lvert_{[-1/4,1/4]}\in H
    \right).
  \end{equation}
  Further, since $(B_0+B_1)$ and $(B_0-B_1)$ are independent Brownian motions, the conditioning on $\underline{m}$ does not affect $(B_0-B_1)$ which still as the law of a Brownian motion of diffusivity $2$. This points towards the truth of \eqref{eq:71} and we now make this formal. Define
  \begin{equation}
    \label{eq:34}
     \underline{\fB}' = ( B_0+B_1,B_0-B_1, \sqrt{2}B_2),
  \end{equation}
  and for ease of notation, define the measurable set $A$ in terms of $F$ such that
$\{\underline{\fB}\in F\}=\{\underline{\fB}'\in A\}$.
For any $\underline{m}\in [1/4,3/4]$, we now have
   \begin{align}
    \label{eq:104}
   &\PP
    \left(\left(
      \underline{\fB}(\underline{m}+\cdot)-\underline{\fB}(\underline{m})\right)\lvert_{[-1/4,1/4] }\in F\lvert\underline{m}
   \right) \nonumber\\
   &=\PP
    \left(\left(
      \underline{\fB}'(\underline{m}+\cdot)-\underline{\fB}'(\underline{m})\right)\lvert_{[-1/4,1/4] }\in A\lvert\underline{m}
   \right)\nonumber\\
   &=\EE\left[\PP
            \left(
            \left(
            \underline{\fB}'(\underline{m}+\cdot)-\underline{\fB}'(\underline{m})\right)\lvert_{[-1/4,1/4] }\in A\lvert \underline{m},B_0-B_1,\sqrt{2}B_2
            \right)\lvert \underline{m}\right]\nonumber\\
   &\leq C \PP
     \left(
     (\sqrt{2}R(\cdot),(B_0-B_1)(\underline{m}+\cdot),\sqrt{2}B_2(\underline{m}+\cdot))\lvert_{[-1/4,1/4]}\in A\lvert \underline{m}
     \right)\nonumber\\
   &=C\PP\left((-B-R,B-R,\sqrt{2}B')\lvert_{[-1/4,1/4]}\in F\right)=C\psi.
    \end{align}
    
To obtain the fourth line, we use \eqref{eq:72}. %
This finishes the proof of \eqref{eq:71}. Now to obtain \eqref{eq:70} from \eqref{eq:71}, we transfer the result using Lemma \ref{cor:43}, and this completes the proof of \eqref{eq:70}.

  It remains to use \eqref{eq:70} to obtain the statement of the lemma. We call an $i\in [\![0,2C'\log^{1/3}\delta^{-1}]\!]$ good if $(m_i-a_i),(b_i-m_i)\in [1/4,3/4]$. Now define the measurable set $F$ from \eqref{eq:110}, \eqref{eq:70} in terms of the set $E$ from the statement of the lemma such that for functions $\zeta_0,\zeta_1,\zeta_2\colon [-1/4,1/4]\rightarrow \RR$, $(\zeta_0,\zeta_1,\zeta_2)\in F$ if and only if $(\zeta_1,\zeta_2)\in E$ and note that this choice implies $\psi=\phi$, where $\phi$ is as in the statement of the lemma. On applying \eqref{eq:70} and doing a simple union bound over all $i$, we obtain
  \begin{align}
    \label{eq:76}
    &\PP
      \left(  (\fB(m_i+\cdot)-\fB(m_i))\lvert_{[-1/4,1/4]}\in F  \text{ for all good } i        
    \right)\nonumber\\
    &\leq 2C'(\log^{1/3}\delta^{-1})(\psi^{1-\nu}\delta^{-\nu}+\delta^{1-\nu}).
  \end{align}
  Now recall that $\{|\Gamma_0(t)|\leq C'\log^{1/3}\delta^{-1}\}\subseteq \{\exists \text{ a good } i \text{ with } \Gamma_0(t)\in I_i\}$. and note that on the latter event, almost surely, $\Gamma_0(t)=m_i$ for the same good $i$ as well. Also, on the event $\{\Gamma_0(t)=m_i\}$, we have
  \begin{displaymath}
    \label{eq:112}
    (\fB(m_i+\cdot)-\fB(m_i))\lvert_{[-1/4,1/4]}\in F \text{ if and only if }(\cR^{t,\theta}-\cR^{t,\theta}(0))\lvert_{[-1/4,1/4]}\in E,
  \end{displaymath}
  and as a consequence of the above, there is a constant $C_1$ such that
  \begin{align*}
    &\PP((\cR^{t,\theta}-\cR^{t,\theta}(0))\lvert_{[-1/4,1/4]}\in E)\\
    &\leq \PP(|\Gamma_0(t)|> C'\log^{1/3}\delta^{-1})+ \PP
      \left(  (\fB(m_i+\cdot)-\fB(m_i))\lvert_{[-1/4,1/4]}\in F  \text{ for all good } i        
                                                   \right)\\
    &\leq \delta + 2C'(\log^{1/3}\delta^{-1})(\psi^{1-\nu}\delta^{-\nu}+\delta^{1-\nu})\\
                                                                                   &=\delta+2C'(\log^{1/3}\delta^{-1})(\phi^{1-\nu}\delta^{-\nu}+\delta^{1-\nu})\\
    &\leq C_1(\phi^{1-\nu}\delta^{-2\nu}+\delta^{1-2\nu}),
  \end{align*}
  where we used the definition of $F$ in terms of $E$ to obtain the second line. For the third line, we used \eqref{eq:74}. For the last line, we just used $\PP\left((B-R,\sqrt{2}B')\lvert_{[-1/4,1/4]}\in E\right)=\phi$ and that $\log^{1/3}\delta^{-1}=o(\delta^{-\nu})$. This completes the proof.
\end{proof}

\printbibliography
\end{document}